\documentclass[a4paper,11pt]{article}
 
\usepackage{latexsym}
\usepackage{amsmath}
\usepackage{amssymb}
\usepackage{amsthm}
\usepackage{psfrag}

\title{
M-convex Function Minimization \\
Under L1-Distance Constraint
}

\author{Akiyoshi Shioura%
\footnote{
Department of Industrial Engineering and Economics,
Tokyo Institute of Technology,
Tokyo 152-8550, Japan
  \texttt{shioura.a.aa@m.titech.ac.jp}}
}

\newcommand{\suppp}{{\rm supp}\sp{+}}
\newcommand{\suppm}{{\rm supp}\sp{-}}

\def\0{\mbox{\bf 0}}
\def\1{\mbox{\bf 1}}

\def\phi{\varphi}
\def\epsilon{\varepsilon}

\def\R{{\mathbb{R}}}

\def\Z{{\mathbb{Z}}}
\def\Oh{{\rm O}}

\def\Mnat{{M$^\natural$}}

\def\Rinf{\R \cup \{+ \infty\}}

\def\P*{{\rm P}_*}

\newcommand{\dom}{{\rm dom\,}}

\def\mycenter{x_{\rm c}}
\def\xcirc{x^\circ}
\def\xbullet{x^\bullet}

\newtheorem{theorem}{Theorem}[section]
\newtheorem{lemma}[theorem]{Lemma}
\newtheorem{corollary}[theorem]{Corollary}
\newtheorem{proposition}[theorem]{Proposition}

\newtheorem{remark}[theorem]{Remark}

\makeatletter
	
	\@addtoreset{equation}{section}
\makeatother

\usepackage{amsmath}	

\begin{document}

\maketitle

\begin{abstract}
 In this paper we consider a new
problem of minimizing an M-convex function under L1-distance constraint (MML1);
the constraint is given by 
an upper bound for L1-distance between a feasible solution
and a given ``center.''
 This is motivated by a nonlinear integer programming problem  for
re-allocation of dock capacity
in a bike sharing system discussed by Freund et al.~(2017).
 The main aim of this paper is to better understand the combinatorial structure  of
the dock re-allocation problem through the connection with M-convexity, and 
show its polynomial-time solvability using this connection.
 For this, we first show that
the dock re-allocation problem can be reformulated in the form of (MML1).
 We then present a pseudo-polynomial-time algorithm for (MML1)
based on steepest descent approach.
 We also propose two polynomial-time algorithms for (MML1) by replacing
the L1-distance constraint with a simple linear constraint.
 Finally, we apply the results for (MML1) to the dock re-allocation problem 
to obtain a pseudo-polynomial-time steepest descent algorithm
and also  polynomial-time algorithms for this problem.
 The proposed algorithm is based on
a proximity-scaling algorithm for a relaxation of 
the dock re-allocation problem, which is of interest in its own right. 
\end{abstract}

\section{Introduction}
\setcounter{equation}{0}

 The concepts of M-convexity and \Mnat-convexity
for functions in integer variables
play a primary role in the theory of
discrete convex analysis \cite{Murota03book}.
 M-convex function, introduced by Murota \cite{Mstein,Mdca}, 
is defined by a certain exchange axiom (see Section~2 for a precise definition),
and enjoys various nice properties as ``discrete
 convexity''
such as a local characterization for global minimality, 
extensibility to ordinary convex functions, conjugacy, duality, etc.
 \Mnat-convex function is introduced by
Murota and Shioura \cite{MS99} as a variant of M-convex function. 
 While the class of \Mnat-convex functions
properly contains that of M-convex functions,
the concept of \Mnat-convexity is essentially equivalent to 
M-convexity in some sense (see, e.g., \cite{Murota03book}). 
 Minimization of an M-convex function is 
the most fundamental optimization problem concerning M-convex functions,
and a common generalization of
the separable convex resource
allocation problem under a submodular constraint and some classes of
nonseparable convex function minimization on integer lattice points. 
  M-convex function minimization 
can be solved by a steepest descent algorithm 
(or greedy algorithm) that runs in pseudo-polynomial time
\cite{Murota03book,Murota03},
and various polynomial-time algorithms have been 
proposed~\cite{MMS2002,Sh1998,Shioura04,T-scal}.

 In this paper, we consider a new problem of minimizing an M-convex function
under the L1-distance constraint, which is formulated as follows:
\[
 \begin{array}{l|lll}
\mbox{(MML1)}& \mbox{Minimize}  & f(x) &  \\
& \mbox{subject to} 
& \sum_{i=1}^n x(i) = \theta, \\
& & \|x - \mycenter \|_1 \le 2\gamma, \\
& 
 & x \in \dom f,
 \end{array}
\]
where $\theta, \gamma \in \Z$,
$f: \Z^n \to \Rinf$ is an M-convex function
such that $\sum_{i=1}^n x(i) = \theta$ holds for 
every $x \in \Z^n$ with $f(x) < + \infty$,
and $\mycenter$ is a vector  (called  the ``center'')
with $f(\mycenter) < + \infty$ and $\sum_{i=1}^n \mycenter(i) = \theta$.
 This problem is motivated by
a nonlinear integer programming problem  for
re-allocation of dock-capacity in a bike sharing system \cite{FHS2017}.

 In a bike sharing system, many bike stations are located
around a city so that  users can rent and return bikes there.
 Each bike station has several docks and bikes;
some docks are equipped with bikes,
and the other docks are kept open
so that users can return bikes at the station.
 The numbers of docks with bike and of open docks
change as time passes, and
it is possible that some users cannot rent or return a bike at a station
due to the shortage of bikes or open docks,
and in such situation users feel dissatisfied.
 To reduce users' dissatisfaction, operators of a bike sharing system
need to re-allocate docks (and bikes) among bike stations appropriately.
 Change to a new allocation, however, 
requires the movement of docks and bikes, which
yields some amount of cost.
 Therefore, it is desirable that 
a new allocation is not so different from the current allocation.
 Hence, the task of 
operators in a bike sharing system is to
minimize users' dissatisfaction by changing the allocation of docks,
while bounding the number of docks to be moved in the re-allocation.

 This problem, which we refer to as
the \textit{dock re-allocation problem},
is discussed by  Freund, Henderson, and Shmoys~\cite{FHS2017} 
and 
formulated  as follows%
\footnote{
While the first constraint is given as an inequality
$\sum_{i=1}^n (d_i + b_i)\le D+ B$ in \cite{FHS2017},
it is implicitly assumed in \cite{FHS2017}
that the inequality holds with equality.
Indeed, the algorithm in \cite{FHS2017}
applies only to the problem with the equality constraint.
}:
\[
 \begin{array}{l|lll}
\mbox{(DR)} & \mbox{Minimize }  & 
 \sum_{i=1}^n c_i(d(i), b(i))\\
& \mbox{subject to } & 
  \sum_{i=1}^n (d(i) + b(i)) = D+ B,\\
& & 
  \sum_{i=1}^n  b(i) \le  B,\\
& & 
  \sum_{i=1}^n  
|(d(i)+b(i)) - (\bar{d}(i) + \bar{b}(i))| \le   2\gamma,\\
& & \ell(i) \le d(i) + b(i) \le u(i), \ 
 d(i), b(i) \in \Z_+\quad  (i \in N).
 \end{array}
\]
 Here, $n \in \Z$ denotes the number of bike stations
and $N = \{1, 2, \ldots, n\}$.
 For a station $i \in N$,
we denote by $b(i), {d}(i) \in \Z_+$, respectively,
the decision variables representing
the numbers of docks with bike and of open docks
allocated at the station.
 The expected  number of dissatisfied users at the station $i$
is represented by a function $c_i: \Z^2_+ \to \R$
in variables $d(i)$ and $b(i)$,
and shown to have the property of \textit{multimodularity}
(see Section~2 for the definition). 

 The first constraint in (DR) means that the total number
of docks (i.e., docks with bike and open docks)
is equal to a fixed constant $D + B$.
 The second constraint gives an upper bound for the total number of docks with bike.
 The third constraint, given in the form of
L1-distance constraint, means that  the difference
between the current and the new allocations of docks should be small,
where $\bar{d}(i)$ and $\bar{b}(i)$ denote, respectively,
the numbers of docks with bike and of open docks
at the station $i$ in the current allocation.
 In addition, 
the number of docks $d(i)+b(i)$ at each station $i$ should be
between 
lower and upper bounds  $[\ell(i), u(i)]$, as represented by
the fourth constraint.

 For the problem (DR),  Freund et al.~\cite{FHS2017} propose
a steepest descent (or greedy) algorithm that repeatedly update 
a constant number of variables by $\pm 1$, and prove by using
the multimodularity of the objective function that
the algorithm finds an optimal solution of (DR) in at most $\gamma$ iterations.
 Hence, the problem (DR) can be solved in pseudo-polynomial time,
while it is not known so far whether 
(DR) can be solved in polynomial time.


{\bf Our Contribution} \qquad
 The main aim of this paper is to 
better understand the combinatorial structure of the problem (DR)
through the connection with M-convexity, 
and to provide polynomial-time algorithms for (DR)
by using the connection.

 We first show that  the dock re-allocation problem (DR) can 
be reformulated in the form of 
the minimization of an M-convex function under the L1-distance constraint (MML1),
where we regard $d(i)+b(i)$ as a single variable
(see Section~\ref{sec:DR-Mconv} for details).

 We then consider the problem (MML1)
and present a steepest descent algorithm that runs in
pseudo-polynomial time.
 While it is known that unconstrained 
M-convex function minimization (without the L1-distance constraint)
can be solved by a certain steepest descent algorithm
(see \cite{Murota03book,Murota03}; see also Section~\ref{sec:sda-MML1} for details),
a naive application of the algorithm does not work for 
the problem (MML1), due to the L1-distance constraint.
 Nevertheless, we prove in Section~\ref{sec:sda-MML1} that
if the center $\mycenter$ is used as an initial solution of the algorithm,
then the steepest descent algorithm finds an optimal solution
in $\gamma$ iterations.
 Moreover, we prove a stronger statement
that for each $k = 0,1,2,\ldots$,
the vector generated in the $k$-th iteration of the steepest descent
algorithm is an optimal solution of 
the M-convex function minimization under the constraint
 $\|x - \mycenter\|_1 = 2k$.
As a byproduct of this result, we obtain new properties
of the steepest descent algorithm for unconstrained M-convex function
minimization.
 In particular, we provide a nontrivial tight bound on the number of 
iterations required by the algorithm,
and show that 
the trajectory of the solutions generated by the algorithm
is a geodesic (i.e, a ``shortest'' path) to the nearest optimal solution
from the initial solution.

 While the problem (MML1) can be solved by a steepest descent algorithm, 
its running time is pseudo-polynomial time.
 To obtain faster algorithms, we present in Section~\ref{sec:poly-MML1} 
two approaches to solve (MML1) in polynomial time.
 For this, we show that by using a minimizer of the M-convex objective function,
the L1-distance constraint in (MML1) can be replaced with
a simple linear constraint;
the two approaches proposed in this section
solve the M-convex function 
minimization under the simple linear constraint instead of the original problem.
 The first approach is to reduce the problem
to the minimization of the sum of two M-convex functions,
for which polynomial-time algorithms are available.
 The second approach is based on the reduction to 
the minimization of another M-convex function
with smaller number of variables,
and the resulting algorithm is faster than the first approach.

 Finally, in Section~\ref{sec:drp} we apply 
the algorithms for (MML1) presented in Sections \ref{sec:sda-MML1} and 
\ref{sec:poly-MML1} 
to the dock re-allocation problem (DR), which can be regarded as
a special case of (MML1).
 We aim at obtaining fast algorithms
by making use of the special structure of (DR).

 In Section~\ref{sec:sda-drp}, we discuss an application
of the steepest descent algorithm in Section~\ref{sec:sda-MML1} to (DR).
 A naive application of the algorithm takes
${\Oh}(n^3 \log (B /n))$ time in each iteration
since it requires ${\Oh}(n \log (B /n))$ time for the 
evaluation of the M-convex function $f$ used in 
the reformulation of (DR).
 To reduce the time complexity, 
we present a useful property of the M-convex function $f$ 
 that the update of function value $f(x)$ can be done
quickly in $\Oh(\log n)$ time if the vector $x$ is updated 
to a vector in a neighborhood. 
 Furthermore, we make full use of this property
to implement the steepest descent algorithm
so that the algorithm works for the original formulation 
and each iteration requires $\Oh(\log n)$ time only.
 We also discuss the connection with
 the steepest descent algorithm  in \cite{FHS2017}.

 Section~\ref{sec:poly-drp} is devoted to polynomial-time algorithms
for (DR).
  While the polynomial-time solvability of (DR) follows 
from the results in Section~\ref{sec:poly-MML1},
a naive application of an algorithm in Section~\ref{sec:poly-MML1}  leads to
a polynomial-time but rather slow algorithm  for (DR);
a faster implementation is difficult this time
since the algorithms in Section~4 are more involved.
 Instead, we use an idea in Section~\ref{sec:poly-MML1}  and the structure of (DR)
to obtain a faster polynomial-time algorithm.
 For this, we replace the L1-distance constraint in (DR)
with a simple linear constraint, as in Section~\ref{sec:poly-MML1}.
 This new formulation, together with the use of  a new problem parameter,
makes it possible to decompose the problem (DR) 
into two independent subproblems, both of which can be
reduced to M-convex function minimization 
and therefore can be solved efficiently.
 We show that an algorithm based on this approach runs in 
$\Oh(n \log n \log ((D+B)/n) \log B)$ time.
 To obtain this time bound, we prove a proximity theorem
for a relaxation of the problem (DR) and devise a
proximity-scaling algorithm for the relaxation; 
the proximity theorem and the algorithm are of interest in their own right.

Most of proofs are provided in Appendix.


\section{Preliminaries on M-convexity}


 Throughout the paper, 
let $n$ be a positive integer with $n \geq 2$ and put $N = \{1, 2, \ldots, n\}$.
 We denote by $\R$
the sets of real numbers, and
by $\Z$ (resp., by $\Z_+$) 
the sets of integers (resp., nonnegative integers);
 $\Z_{++}$ denotes the set of positive integers.

 Let $x = (x(1), x(2), \ldots, x(n)) \in \R^n$ be a vector.
 We denote 
$\suppp(x) = \{i \in N \mid x(i)> 0\}$ and
$\suppm(x) = \{i \in N \mid x(i) < 0\}$.
 For a subset $Y \subseteq N$, we denote $x(Y) = \sum_{i\in Y}x(i)$.
 We define
$\|x\|_1 = \sum_{i\in N} |x(i)|$
and $\|x\|_\infty = \max_{i\in N} |x(i)|$.

 We define
$\0 = (0, 0, \ldots, 0) \in \Z^n$.
 For $Y \subseteq N$, we denote by 
$\chi_Y \in\{0, 1\}^n$ the characteristic vector of $Y$,
i.e., $\chi_Y(i) = 1$ if $i \in Y$ and 
$\chi_Y(i) = 0$ otherwise.
 In particular, we denote $\chi_i = \chi_{\{i\}}$
for every $i \in N$.
 We also denote $\chi_0 = \0$.
 Inequality $x \leq y$
for vectors $x, y \in \R^n$ means component-wise inequality
$x(i) \leq y(i)$ for all $i \in N$.

\subsection{M-convex and Multimodular Functions}

 Let $f: \Z^n \to \Rinf$ be a function.
 The {\it effective domain} of $f$
is defined by $\dom f= \{x \in \Z^n \mid f(x) < + \infty\}$,
and the set of minimizers of $f$ is denoted by
$\arg\min f$.
 Function $f$
is said to be \textit{\Mnat-convex} 
if it satisfies the following exchange property:
\begin{quote}
{\bf (\Mnat-EXC)}
$\forall x, y \in \dom f$, $\forall i \in \suppp(x-y)$,
$\exists j \in \suppm(x-y)\cup\{0\}:$
\[
 f(x)  + f(y) \ge f(x - \chi_i + \chi_j) + f(y + \chi_i - \chi_j).
\]
\end{quote}
 For an \Mnat-convex function  $f$,
if  $\dom f$ is contained in a hyperplane
$\{x \in \Z^n \mid x(N) = \theta\}$ for some $\theta \in \Z$,
then $f$ is called an \textit{M-convex function}, in particular.
 It is known that a function $f$ is M-convex if and only if
it satisfies the following exchange property:
\begin{quote}
{\bf (M-EXC)}
$\forall x, y \in \dom f$, $\forall i \in \suppp(x-y)$,
$\exists j \in \suppm(x-y):$
\[
 f(x)  + f(y) \ge f(x - \chi_i + \chi_j) + f(y + \chi_i - \chi_j).
\]
\end{quote}

M-/\Mnat-convex functions can be characterized by
seemingly weaker exchange properties.

\begin{theorem}[{\cite[Theorem~6.4]{Murota03book}, \cite{MS99}}]
\label{thm:M-char-exists}
 Let  $f: \Z^n \to \Rinf$ be a function.\\
{\rm (i)}
$f$ is M-convex if and only if
it satisfies the following condition:\\
\hspace*{5mm}
 $\forall x, y \in \dom f$ with $x \ne y$,
$\exists i \in \suppp(x-y)$,  $\exists j \in \suppm(x-y):$
\[
 f(x)  + f(y) \ge f(x - \chi_i + \chi_j) + f(y + \chi_i - \chi_j).
\]
{\rm (ii)}
$f$ is \Mnat-convex if and only if
it satisfies the following condition:
\\
\hspace*{5mm}
 $\forall x, y \in \dom f$ with $x \ne y$ and $x(N) \ge y(N)$,
$\exists i \in \suppp(x-y)$,  $\exists j \in \suppm(x-y)\cup \{0\}:$
\[
 f(x)  + f(y) \ge f(x - \chi_i + \chi_j) + f(y + \chi_i - \chi_j).
\]
\end{theorem}

\Mnat-convexity of a function implies the following
exchange properties.

\begin{theorem}[\cite{MS99}]
\label{thm:Mnat-convex-ineq}
 Let $f: \Z^n \to \Rinf$ be an \Mnat-convex function and 
$x, y \in \dom f$.
\\
{\rm (i)}
If $x(N)\le y(N)$, then 
for every $i \in \suppp(x-y)$ there exists some $j \in \suppm(x-y)$
such that
\[
 f(x)  + f(y) \ge f(x - \chi_i + \chi_j) + f(y + \chi_i - \chi_j).
\]
{\rm (ii)}
If $x(N) < y(N)$, then 
 there exists some $j \in \suppm(x-y)$
such that
\[
 f(x)  + f(y) \ge f(x  + \chi_j) + f(y  - \chi_j).
\]
\end{theorem}

 We then explain the concept of multimodularity
and its connection with \Mnat-convexity.
 A function $\phi: \Z^2_+ \to \R$ in two variables
 is called \textit{multimodular} if 
it satisfies the following conditions:
\begin{align*}
&
 \phi(\eta+1, \zeta+1) - \phi(\eta+1, \zeta)
 \ge  \phi(\eta, \zeta+1) - \phi(\eta, \zeta)
\qquad (\forall \eta, \zeta \in \Z_+),
\\
& 
\phi(\eta-1, \zeta+1) - \phi(\eta-1, \zeta)  \ge  \phi(\eta, \zeta) - \phi(\eta, \zeta-1)
\qquad
(\forall  \eta, \zeta \in \Z_{++}),
\\
& 
\phi(\eta+1, \zeta-1) - \phi(\eta, \zeta-1)  \ge  \phi(\eta, \zeta) - \phi(\eta-1, \zeta)
\qquad
(\forall  \eta, \zeta \in \Z_{++}).
\end{align*}
\noindent
 For functions  in two variables,
multimodularity and \Mnat-convexity are
essentially equivalent.

\begin{proposition}[{cf.~\cite{MoriMuro2018}}]
\label{prop:f2-Mnat-multi}
 A function $\phi: \Z^2_+ \to \R$ in two variables
 is multimodular if and only if 
the function $f: \Z^2 \to \Rinf$  
given by 
\begin{equation}
\label{eqn:f2-Mnat} 
\dom f = \Z^2_+, \qquad 
f(\alpha,\beta) = \phi(\alpha,\beta) \quad ((\alpha,\beta) \in \dom f)
\end{equation}
is \Mnat-convex.
\qed
\end{proposition}

 This relationship  and  Theorem \ref{thm:Mnat-convex-ineq}
immediately imply the following property of multimodular functions.

\begin{proposition}
\label{prop:2multimod}
 Let $\phi: \Z^2_+ \to \R$ be a multimodular function, and
 $\eta, \zeta, \eta', \zeta' \in \Z_+$.\\
{\rm (i)}
If $\eta> \eta'$ and $\zeta < \zeta'$, then it holds that
\begin{equation}
\label{eqn:2multimod-1}
 \phi(\eta, \zeta)+\phi(\eta', \zeta') 
\ge \phi(\eta-1, \zeta+1)+\phi(\eta'+1, \zeta'-1). 
\end{equation}
{\rm (ii)}
If $\eta> \eta'$ and $\eta+\zeta > \eta'+\zeta'$, then it holds that
\begin{equation}
\label{eqn:2multimod-2}
\phi(\eta, \zeta)+\phi(\eta', \zeta') 
\ge \phi(\eta-1, \zeta)+\phi(\eta'+1, \zeta').
\end{equation}
\end{proposition}

\begin{proof}
By Proposition \ref{prop:f2-Mnat-multi},
 $\phi$ can be seen as an \Mnat-convex function.
 We first prove the claim (i).
Theorem \ref{thm:Mnat-convex-ineq} (i) implies that
if $\eta+\zeta \le \eta'+\zeta'$
then
$\phi(\eta, \zeta)+\phi(\eta', \zeta') 
\ge \phi(\eta-1, \zeta+1)+\phi(\eta'+1, \zeta'-1)$,
and  
if $\eta'+\zeta' \le \eta+\zeta$
then
$\phi(\eta', \zeta') + \phi(\eta, \zeta)
\ge \phi(\eta'+1, \zeta'-1) + \phi(\eta-1, \zeta+1)$.
 In either case, 
the inequality \eqref{eqn:2multimod-1} holds.

 We then prove the claim (ii).
 If  $\zeta \ge \zeta'$, then
the inequality \eqref{eqn:2multimod-2} follows immediately from
Theorem \ref{thm:Mnat-convex-ineq} (ii).
If $\eta+\zeta > \eta'+\zeta'$ and $\zeta < \zeta'$, then
the inequality \eqref{eqn:2multimod-2} follows immediately from
(\Mnat-EXC).
\end{proof}

\subsection{Minimization of an M-convex Function}
\label{sec:Mmin}

 We consider the minimization of an M-convex function.
 A minimizer of an M-convex function can be characterized by 
a local optimality condition.

\begin{theorem}[{cf.~\cite[Theorem~6.26]{Murota03book}}]
 For an M-convex function $f: \Z^n \to \Rinf$,
a vector $x^* \in \dom f$ 
is a minimizer of $f$ if and only if
$f(x^* - \chi_i + \chi_j) \ge f(x^*)$ $(\forall i, j \in N)$.
\end{theorem}

 This theorem immediately implies that
 the  minimization of an 
M-convex function
can be solved by the following steepest descent algorithm
(see, e.g., \cite[Section~10.1.1]{Murota03book}):

\begin{flushleft}
 \textbf{Algorithm} {\sc SteepestDescent} 
\\
 \textbf{Step 0:} 
 Let $x_0 \in \dom f$ be an appropriately chosen initial vector.
 Set $k:=1$.
\\
 \textbf{Step 1:} 
If 
$f(x_{k-1} + \chi_{i} - \chi_{j}) \ge f(x_{k-1})$
for every $i, j \in N$, 
then output $x_{k-1}$ and stop.\\
 \textbf{Step 2:} 
 Find $i_k, j_k \in N$ that minimizes $f(x_{k-1} + \chi_{i_k} - \chi_{j_k})$.\\
 \textbf{Step 3:} 
 Set $x_{k} := x_{k-1} + \chi_{i_k} - \chi_{j_k}$, 
$k:=k+1$, and go to Step 1.
\end{flushleft}

\begin{theorem}[{cf.~\cite[Section~10.1.1]{Murota03book}}]
\label{thm:}
 Let $f: \Z^n \to \Rinf$ be an M-convex function $f: \Z^n \to \Rinf$ 
with bounded $\dom f$.
 Then, the algorithm {\sc SteepestDescent} 
outputs a minimizer of $f$ after a finite number of iterations.
\end{theorem}

  Polynomial-time algorithms based on proximity scaling algorithms are
proposed for M-convex function minimization~\cite{MMS2002,Sh1998,Shioura04,T-scal},
and the current best time complexity bounds are given as follows.
 For a set $S\subseteq \Z^n$, we define the \textit{L$_\infty$-diameter} of $S$ by
\begin{equation}
\label{eqn:def-sizeL}
L = \max\{\|x - y\|_\infty \mid x, y \in S\}. 
\end{equation}

\begin{theorem}[\cite{Shioura04,T-scal}]
\label{thm:M-time}
Minimization of  an M-convex function $f: \Z^n \to \Rinf$
can be done 
in   $\Oh(n\sp{3}\log (L/n)F)$ time, 
where $L$ is the L$_\infty$-diameter of $\dom f$
 and $F$ denotes the time to evaluate the function
value of $f$.
\end{theorem}

 We also consider 
the minimization of an \Mnat-convex function 
$f: \Z^n \to \Rinf$ under the constraint that $x(N) = \theta$
for a given $\theta \in \Z$.
While this problem is essentially equivalent to
 M-convex function minimization, 
it can be solved faster if $\dom f$
is given by an interval.

\begin{theorem}[{cf.~\cite{Shioura04}}]
\label{thm:Mnat-polymat-time}
 Let $f: \Z^n \to \Rinf$ be an \Mnat-convex function
such that $\dom f$ is given by an interval,
and $\theta \in \Z_{++}$.
 Then, the minimization of $f$
under the constraint $x(N)=\theta$
can be  solved   in   $\Oh(n\sp{2}\log (L/n)F)$ time, 
where $L$ is the L$_\infty$-diameter of the set
$\{x \in \dom f \mid x(N)=\theta \}$ and
$F$ denotes the time to evaluate the function value of $f$.
\end{theorem}

\section{Reformulation of Dock Re-allocation Problem as (MML1)}
\label{sec:DR-Mconv}

 We consider the dock re-allocation problem (DR)
explained in Introduction.
 Using vector notation, the problem (DR)
can be simply rewritten as follows:
\[
 \begin{array}{l|lll}
\mbox{(DR)} & \mbox{Minimize }  & 
\displaystyle c(d, b)\\
& \mbox{subject to } & 
\displaystyle  d(N)+b(N) = D+ B,\\
& & 
\displaystyle  b(N) \le  B,\\
& & 
\| (d+b) - (\bar{d} + \bar{b})\|_1
 \le  2\gamma,\\
& & \ell \le d + b \le u,\   d, b \in \Z^n_+,
 \end{array}
\]
where $c: \Z^n_+ \times \Z^n_+ \to \R$ is a function given by
$c(d,b) = \sum_{i=1}^n c_i(d(i), b(i))$ $((d,b) \in \Z^n_+ \times \Z^n_+)$.
 In this section, we show that (DR)
can be reformulated as the problem (MML1).

 We define a function $f: \Z^n \to \Rinf$ by
\begin{align}
& \dom f = \{x \in \Z^n \mid  x(N) = D + B,\ \ell \le x \le u ,\ 
\bar{d} + \bar{b} -\gamma \1 \le   x  \le \bar{d} + \bar{b} + \gamma \1 \},
\notag\\
& f(x) = 
 \min\{ c(d, b) \ \mid
 d, b \in \Z_+^n,\ d + b = x,\ 
   b(N) \le  B \}
\quad (x\in \dom f).
 \label{eqn:def-f}
\end{align}
 As shown below, $f$ is  an \Mnat-convex function.
 With this function $f$, the problem (DR) can be reformulated as
\[
 \begin{array}{l|lll}
& \mbox{Minimize}  & f(x) &  \\
& \mbox{subject to} 
& x(N) = D + B,
\\
& & \|x - (\bar{d}+\bar{b}) \|_1 \le 2\gamma,\\
& & x \in \dom f.
 \end{array}
\]
 Hence, (DR) is reformulated as (MML1).
 We note that in the reformulation of (DR) above, the constraint 
$\| x - (\bar{d} + \bar{b})\|_1  \le  2\gamma$
implies the inequality
$\bar{d} + \bar{b} -\gamma \1 \le    x
 \le \bar{d} + \bar{b} + \gamma \1$
that appears in the definition of $\dom f$  in \eqref{eqn:def-f}.
Hence, addition of this constraint 
in the definition of $\dom f$
is not necessary in the reformulation above, 
but it is added to obtain a better time complexity in the following section.

\begin{theorem}
\label{thm:Mnat-convex-laminar} 
 Function $f$ in {\rm \eqref{eqn:def-f}} is M-convex.
\end{theorem}

 We also consider a function $\hat f: \Z^n \to \Rinf$ by
\begin{align}
& \dom \hat f = \{x \in \Z^n \mid  \ell \le x \le u ,\  
\bar{d} + \bar{b} -\gamma \1 \le   x  \le \bar{d} + \bar{b} + \gamma \1 \},
\notag\\
& \hat f(x) = 
 \min\{ c(d, b) \ \mid
 d, b \in \Z_+^n,\ d + b = x,\ 
   b(N) \le  B \}
\quad (x\in \dom \hat f).
 \label{eqn:def-fhat}
\end{align}
 The difference from the function $f$ in \eqref{eqn:def-f}
is that the equation $x(N) = D+B$ 
is missing in the definition of $\hat{f}$.
 It is easy to see that for every $x \in \Z^n$ with 
$x(N) = D+B$, we have $\hat{f}(x) = f(x)$.
 In a similar way as $f$, we can show that $\hat{f}$
is an \Mnat-convex function.
 Hence, instead of $f$,
we may use $\hat{f}$ as an objective function of the
reformulation of (DR).
 This objective function is useful 
in obtaining a faster algorithm.
 We note that the effective domain of $\hat{f}$
is given by an interval.
 This fact is used in Section~\ref{sec:drp}.


\section{Steepest Descent Algorithm for (MML1)}
\label{sec:sda-MML1}

 In this section, we show that an optimal solution of the problem (MML1)
can be obtained by using a variant of the steepest descent algorithm
{\sc SteepestDescent} in Section~\ref{sec:Mmin}
for unconstrained M-convex function minimization.
 While we are mainly interested in the case where the center $\mycenter$  
is a feasible solution to (MML1), we also consider the case with
infeasible~$\mycenter$.  
 We assume that the effective domain $\dom f$ of 
the function $f$ is bounded; this assumption implies that
 $\arg\min f \ne \emptyset$, in particular.

 Let $\sigma \in \Z_+$ be 
the half of   L1-distance between $\mycenter$
and a nearest vector in $\dom f$,
and 
 $\tau \in \Z_+$
the half of   L1-distance between $\mycenter$
and a nearest minimizer of $f$, i.e.,
\begin{align}
\label{eqn:def-sigma}
\hspace*{-3mm}
  \sigma & = (1/2)\min\{\|x - \mycenter \|_1 \mid x \in \dom f \},
&
  \tau & = (1/2)\min\{\|x - \mycenter \|_1 \mid x \in \arg \min f \}.
\end{align}
 We have $\sigma = 0$ if $\mycenter$ is a a feasible solution.
 Also, note that   a minimizer $x^\bullet$ of $f$
with $\|x^\bullet - \mycenter\|_1 = 2\tau$
is given by a minimizer of a function $f(x) + \epsilon \|x - \mycenter\|_1$
with a sufficiently small positive $\epsilon$.
 Since the sum of an M-convex function and a separable-convex function
 is M-convex \cite[Theorem~6.13]{Murota03book},
a minimizer of $f(x) + \epsilon \|x - \mycenter\|_1$
can be obtained by any algorithm for unconstrained M-convex function minimization.
 If $\tau \le \gamma$, then
the vector $x^\bullet$ is optimal for (MML1).
 Hence, we assume $\tau > \gamma$ in the following.

 In the following, we denote by (MML1$(k)$)
the problem (MML1) with the constant $\gamma$
in the L1-distance constraint is replaced with a parameter $k \in \Z_+$.
 We first present a property of optimal solutions of
the  problem (MML1$(k)$).
 For every  $k$,
we denote by $M_k \subseteq \Z^n$ and by $\mu_k \in \R$, respectively, 
the set of optimal solutions 
and the optimal value of the problem (MML1$(k)$).
 We have $M_0 = \{\mycenter\}$ 
and $\mu_0 = f(\mycenter)$ if $\mycenter$ is feasible;
we also have
$M_k = \{x \in \arg\min f \mid \|x - \mycenter\|_1 \le 2k\}$
and $\mu_k = \min f$ for every $k \ge \tau$.

\begin{theorem}
\label{thm:norm-opt-main}
\ \\
{\rm (i)}
It holds that
$\mu_\sigma > \mu_{\sigma+1} > \cdots > \mu_{\tau}$
and $M_k \subseteq \{x \in \Z^n \mid \|x - \mycenter\|_1 = 2k \}$
for $k \in [\sigma, \tau]$.
\\
{\rm (ii)}
For every integer $k \in [\sigma, \tau-1]$
and $y \in M_k$,
there exists some $\tilde y \in M_{k+1}$ 
such that 
$\tilde y = y + \chi_i - \chi_j$ for some 
$i \in N \setminus \suppm(y - \mycenter)$
and $j \in N \setminus \suppp(y - \mycenter)$.
\\
{\rm (iii)}
 For every integer $k \in [\sigma, \tau-1]$  and $y \in M_{k+1}$,
 there exists some $y' \in M_{k}$ such that  
$y' = y - \chi_i + \chi_j$ for some 
$i \in \suppp(y - \mycenter)$
and $j \in \suppm(y - \mycenter)$.
\end{theorem}

\noindent
This is the key property to prove the
validity of the algorithms presented in this section. 
 In particular, we see from the claim (i) in the theorem that
the set of optimal solutions of (MML1) is given by  $M_\gamma$.

 Theorem \ref{thm:norm-opt-main}  implies that
 a variant of  the steepest descent algorithm 
for unconstrained M-convex function minimization
finds an optimal solution of (MML1).

\begin{flushleft}
 \textbf{Algorithm} {\sc SteepestDescentMML1} 
\\
 \textbf{Step 0:} 
 Compute  $\sigma$ in \eqref{eqn:def-sigma}
and  $\xcirc \in M_\sigma$.
 Set $x_\sigma := \xcirc$ and $k:=\sigma + 1$.
\\
 \textbf{Step 1:} 
If  $k-1 = \gamma$,
then output $x_{k-1}$ and stop.\\
 \textbf{Step 2:} 
 Find $i_k, j_k \in N$ that minimizes $f(x_{k-1} + \chi_{i_k} - \chi_{j_k})$.\\
 \textbf{Step 3:} 
 Set $x_{k} := x_{k-1} + \chi_{i_k} - \chi_{j_k}$, 
$k:=k+1$, and go to Step 1.
\end{flushleft}

\begin{theorem}
\label{thm:main1}
The algorithm {\sc SteepestDescentMML1} applied to
an M-convex function $f: \Z^n \to \Rinf$
outputs an optimal solution of {\rm (MML1)}
in $\gamma-\sigma$ iterations.
 Moreover, the  vector $x_k$ 
generated in each iteration of the algorithm
satisfies $x_k \in M_{k}$.
\end{theorem}

\begin{proof}
 We prove by induction that
$x_k \in M_k$ for each $k$.
 Assume that $x_{k-1} \in M_{k-1}$ holds for some $k < \gamma$.
 By the behavior of the algorithm and Theorem \ref{thm:norm-opt-main},
$x_k$ is given as 
$x_k= x_{k-1} + \chi_{i_k} - \chi_{j_k}$  
with $i_k \ne j_k$ and satisfies
 $x_k \in M_k$.
\end{proof}
\noindent
 Note that the running time of the algorithm {\sc SteepestDescentMML1},
except for Step 0, is $\Oh(n^2 (\gamma - \sigma))$, provided that the evaluation of
function value can be done in constant time.
 Computation of $\sigma$ and $\xcirc$
in Step 0 can be done by 
finding a minimizer $\xcirc$ of a function $f(x) + \Upsilon \|x - \mycenter\|_1$
with a sufficiently large positive $\Upsilon > \max\{f(x) \mid x \in \dom f\}$
and then setting $\sigma = \|\xcirc - \mycenter\|_1$.
 Function $f(x) + \Upsilon \|x - \mycenter\|_1$ is also M-convex, 
and therefore its minimization can be done by any algorithm
for M-convex function minimization, even if
the value $\Upsilon$ is not given specifically.

Using Theorem \ref{thm:norm-opt-main} (iii),
we can also consider another variant of steepest descent algorithm
that starts from a nearest minimizer $x^\bullet$ of $f$
and greedily approaches $\mycenter$; see Appendix.



\section{Polynomial-Time Algorithms for (MML1)}
\label{sec:poly-MML1}

 In this section we show that the 
problem {\rm (MML1)} can be solved in polynomial time. 
 As in Section~\ref{sec:sda-MML1}, we assume that 
the value $\tau$ in \eqref{eqn:def-sigma} satisfies
$\tau > \gamma$,
and let 
$x^\bullet \in \dom f$ be a minimizer of $f$
with $\|x^\bullet - \mycenter\|_1 = 2\tau$, which is fixed throughout 
this section.

 We note that every vector $x$ 
satisfying the constraint $\|x - \mycenter \|_1 \le 2\gamma$
is contained in the interval $[\mycenter - \gamma \1, \mycenter + \gamma \1]$.
 Hence, we assume in this section 
that the effective domain $\dom f$ of  $f$
is also contained in the interval $[\mycenter - \gamma \1, \mycenter + \gamma \1]$;
if the given $f$ does not satisfy this condition, then
it suffices to consider the restriction of $f$ on this interval.
 This assumption implies that the L$_\infty$-diameter of $f$ is bounded by $2\gamma$;
we use this fact in the analysis of algorithms.

\subsection{Reduction to Problem with Linear Constraints}

 We first show that
the L1-distance constraint $\|x - \mycenter\|_1 \le 2 \gamma$
in (MML1) can be replaced with a system of linear constraints.
 Let us consider the following problem:
\[
 \begin{array}{l|lll}
\mbox{(MM-L)} & \mbox{Minimize}  & f(x) &  \\
& \mbox{subject to} 
& x(N) = \theta,
\\
& & 
x(P) = \mycenter(P) + \gamma,
\\
& &  
\hat{\ell} \le x \le \hat{u},
\\
& & x \in \dom f,
 \end{array}
\]
where  $P = \suppp(x^\bullet - \mycenter)$,
and $\hat{\ell}, \hat{u} \in \Z^n$ are vectors given by
\[
\hat{\ell}(i) = 
\begin{cases}
\mycenter(i) &  (i \in P),\\
\max\{\xbullet(i), \mycenter(i)  - \gamma\} & (i \in N \setminus P),
\end{cases}
\quad
\hat{u}(i) = 
\begin{cases}
\min\{\xbullet(i), \mycenter(i)  + \gamma\} & (i \in P),\\
\mycenter(i) &  (i \in N \setminus P).
\end{cases}
\]

\begin{lemma}
\label{lem:MML-MML1}
Every optimal solution of  {\rm (MM-L)} 
is also optimal for  {\rm (MML1)}.
\end{lemma}

 While the problem (MM-L) does not fit into the framework
of M-convex function minimization problem,
due to the constraint $x(P) = \mycenter(P) + \gamma$,
it can be formulated
as the minimization of the sum of two M-convex functions.
 Indeed, (MM-L) is equivalent to the minimization of
the sum of functions $f_1, f_2: \Z^n \to \Rinf$  given by
\begin{align*}
f_1(x)  
& = 
\begin{cases}
f(x)  & (\mbox{if }
x(N) = \theta),\\
+ \infty & (\mbox{otherwise}),
\end{cases}
\\
f_2(x)  
& = 
\begin{cases}
0 & (\mbox{if }x(N) = \theta,\ 
x(P) = \mycenter(P) + \gamma,\ \hat{\ell} \le x \le \hat u
),\\
+ \infty & (\mbox{otherwise}).
\end{cases}
\end{align*}
 It is not difficult to see that $f_1$ and $f_2$ satisfy (M-EXC), i.e.,
the two functions are M-convex.

 It is known that minimization of the sum of two \Mnat-convex functions
$f_1, f_2: \Z^n \to \Rinf$
can be solved in polynomial time (see, e.g., \cite{Murota03book}),
and the fastest algorithm runs in $\Oh(n^6 (\log L)^2 \log(nK))$ time~\cite{IMM05},
where $L$ is the maximum of the L$_\infty$-diameter of $\dom f_1$ and of $\dom f_2$
(see \eqref{eqn:def-sizeL} for the definition of L$_\infty$-diameter)
and 
$K$ is given by
$K = \max_{h=1,2}  \max\{|f_h(x) - f_h(y)| \mid x, y \in \dom f_h \}$.
 For the functions $f_1$ and $f_2$ defined above,
the L$_\infty$-diameter of $f_1$ and $f_2$ is bounded by 
$\max_{i \in N}\{\hat u(i)-\hat{\ell}(i)\} \le \gamma$.
 Hence, 
we obtain the following result.

\begin{theorem}
The problem {\rm (MML1)}
 can be solved in $\Oh(n^6 (\log \gamma)^2 \log(nK_f))$ time,
 where 
 $K_f = \max\{|f(x) - f(y)| \mid x, y \in \dom f\}$.
\end{theorem}

\subsection{Reduction to M-convex Function Minimization}

 We now explain an alternative approach to solve
the problem (MM-L) by the reduction to the minimization
of an M-convex function.

 For a vector $y \in \Z^{N\setminus P}$ we define a set $T(y) \subseteq \Z^n$
by
\[
 T(y)  = \{x \in \dom f \mid 
 x(i)=y(i)\ (i \in N \setminus P),\  
\hat{\ell}(i) \le x(i) \le \hat{u}(i) \ (i \in P)
 \}.
\]
 Then, the function $g: \Z^{N\setminus P} \to \Rinf$ is defined as follows:
\begin{align}
g(y) &=
\begin{cases}
 \min\{f(x) \mid x \in T(y)\}
& (\mbox{if }
y(N \setminus P) = \theta - (\mycenter(P) + \gamma) 
\\
& 
\hspace*{15mm}
\mbox{ and } 
\hat{\ell}(i) \le y(i) \le \hat{u}(i) \ (\forall i \in N \setminus P)),
\\
 + \infty & (\mbox{otherwise}).
\end{cases}
\label{eqn:def-g}
\end{align}
 By definition, $x \in \Z^n$  is a feasible solution of (MM-L)
if and only if the vector $y \in \Z^{N \setminus P}$
given by $y(i)=x(i)\ (i \in N \setminus P)$
satisfies $y \in \dom g$ and $x \in T(y)$.
 Therefore,  the problem (MM-L) can be reduced to the minimization
of function $g$;  for a minimizer $y^*\in \Z^{N \setminus P}$ of $g$,
the vector  $x^* \in T(y^*)$ with $g(y^*)= f(x^*)$ 
is an optimal solution of (MM-L).

\begin{proposition}
\label{prop:g-Mnat-polymat}
Function $g$ is  M-convex.
\end{proposition}

 We analyze the running time of the algorithm.
 By Theorem \ref{thm:M-time},
the minimization of $g$
can be done 
in   $\Oh(n\sp{3}\log (\gamma/n)F_g)$ time, 
where $F_g$ denotes the time to evaluate the function value of $g$.
 The evaluation of the value of function $g$
can be seen as the minimization of an M-convex function.
 Since the L$_\infty$-diameter of $f$ is bounded by $\gamma$,
the evaluation of $g$ can be done 
in   $\Oh(n\sp{3}\log (\gamma/n))$ time 
by Theorem \ref{thm:M-time},
provided that the function evaluation of $f$ can be done in constant time.
 Hence, we obtain the following time complexity result:

\begin{theorem}
\label{thm:MML1-polytime}
The problem {\rm (MML1)}
can be solved  
in   $\Oh(n\sp{6}(\log (\gamma/n))^2)$ time. 
\end{theorem}


\section{Application to Dock Re-allocation Problem}
\label{sec:drp}

 As observed in Section~\ref{sec:DR-Mconv}, the dock re-allocation problem (DR)
can be seen as a special case of the problem (MML1).
 In this section, we apply the results obtained in Sections 
\ref{sec:sda-MML1} and 
\ref{sec:poly-MML1} for (MML1) to obtain algorithms for (DR).
 In particular, we show that the problem (DR) can be solved in polynomial time.

\subsection{Steepest Descent Algorithm}
\label{sec:sda-drp}

 We first present a steepest descent algorithm for (DR)
by applying the algorithm in Section~\ref{sec:sda-MML1} for (MML1).
 We also show that a fast implementation of the steepest descent algorithm
coincides with the greedy algorithm proposed by Freund et al.~\cite{FHS2017}

 Recall that (DR) can be reformulated in the form of (MML1) as
\[
 \begin{array}{l|lll}
& \mbox{Minimize}  & f(x) &  \\
& \mbox{subject to} 
& x(N) = D + B,
\ \|x - (\bar{d}+\bar{b}) \|_1 \le 2\gamma,\ 
 x \in \dom f,
 \end{array}
\]
where the M-convex function $f: \Z^n \to \Rinf$ is given by
\eqref{eqn:def-f}.
 By definition, the function value $f(x)$ for a given $x \in \dom f$ 
can be computed by solving the following problem:
\[
 \begin{array}{l|lll}
\mbox{(SRA($x$))}
 & \mbox{Minimize }  & 
 c(x-b, b) \equiv
\sum_{i=1}^n c_i(x(i) - b(i), b(i))
\\
& \mbox{subject to } & 
\displaystyle  b(N) \le  B,\  \0 \le b \le x,\   b \in \Z^n.
 \end{array}
\]
 It is observed that 
for each $i \in N$, 
$c_i(x(i) - b(i), b(i))$ is a convex function in 
variable $b(i)$ since $c_i$ is a multimodular (or \Mnat-convex) function.
 Hence, the problem (SRA$(x)$) can be seen as a simple resource allocation
problem and therefore the evaluation of the function value of $f$ can be done
in $\Oh(n \log (B/n))$ time (see, e.g., \cite{Hoch94}).

 The algorithm {\sc SteepestDescentMML1} 
is rewritten in term of the problem (DR) 
as follows.
 Recall that $(\bar{d}, \bar{b})$ is a feasible solution
of the problem (DR), and therefore
the vector $\bar{x} = \bar{d} + \bar{b}$ is used as the initial
solution of the steepest descent algorithm.

\begin{flushleft}
 \textbf{Algorithm} {\sc SteepestDescentDR} 
\\
 \textbf{Step 0:} 
 Set $x_0 := \bar{d} + \bar{b}$ and $k:= 1$.
\\
 \textbf{Step 1:} 
If  $k-1 = \gamma$,
then output the solution $(x_{k-1}-b_{k-1}, b_{k-1})$ and stop.\\
 \textbf{Step 2:} 
 For every distinct $i,j \in N$, compute the 
value $f(x_{k-1} + \chi_{i} - \chi_{j})$
by solving  
\\
\noindent
\phantom{Step 1: }
 (SRA$(x_{k-1} + \chi_{i} - \chi_{j})$),
and find $i_k, j_k \in N$ minimizing   $f(x_{k-1} + \chi_{i_k} - \chi_{j_k})$.
\\
 \textbf{Step 3:} 
Let $b_k$ be an optimal solution of (SRA$(x_{k-1} + \chi_{i_k} - \chi_{j_k})$),
set 
\\
\noindent
\phantom{Step 1: }
$x_{k} := x_{k-1} + \chi_{i_k} - \chi_{j_k}$, $k:=k+1$, and 
go to Step 1.
\end{flushleft}

 Since the evaluation of the  function value $f(x)$
requires $\Oh(n \log (B/n))$ time, 
each iteration requires
$\Oh(n^3 \log (B/n))$ time, and 
the total running time of the algorithm 
is $\Oh(\gamma  n^3 \log (B/n))$.

 The next lemma shows that
the evaluation of the value $f(x)$ can be done faster
by maintaining an optimal solution of the problem (SRA$(x_k)$)
for each $k$.
 This lemma is essentially equivalent to
Lemma 6 in \cite{FHS2017}, while 
the statement of the lemma is described differently in our notation.

\begin{lemma}[{\cite[Lemma~6]{FHS2017}}]
\label{lem:sra-opt-update}
 Let  $x \in \dom f$, and  $b \in \Z^n$ be an optimal  solution of
the problem {\rm  (SRA$(x)$)}. 
 Also, let $i, j \in N$ be distinct elements
such that $x + \chi_i - \chi_j \in \dom f$.
 Then, there exists  an optimal  solution  $\hat b \in \Z^n$ of
the problem {\rm  (SRA$(x + \chi_i - \chi_j)$)} such that
\begin{align}
& \hat{b} \in \{b, b + \chi_i, b - \chi_j, b + \chi_i - \chi_j\}
\notag\\
& \qquad
\cup \{ b + \chi_i - \chi_t \mid  t \in N \setminus \{i,j\}\}
\cup \{ b + \chi_s - \chi_j \mid  s \in N \setminus \{i,j\}\}.
\label{eqn:sra-opt-update-cond1}
\end{align}
\end{lemma}

It follows from Lemma \ref{lem:sra-opt-update} that
for each $i, j \in N$, 
an optimal solution of the problem (SRA$(x + \chi_{i} - \chi_{j})$)
can be found in $\Oh(n)$ time, provided that
an optimal solution of the problem {\rm  (SRA$(x)$)} is available.
 Therefore, 
the running time of the algorithm {\sc SteepestDescentDR}
can be reduced to $\Oh(\gamma \, n^3)$.
 
 In fact, Lemma \ref{lem:sra-opt-update} implies that
the running time $\Oh(n^3)$ in each iteration can be further reduced
by computing elements $i_k, j_k \in N$ 
minimizing the value $f(x_{k-1} + \chi_{i_k} - \chi_{j_k})$
and 
an optimal solution of the problem (SRA$(x + \chi_{i_k} - \chi_{j_k})$)
simultaneously.
 We denote
\[
R = \{(d, b) \in \Z^n \times \Z^n
\mid
 d(N)+b(N) = D+ B,\  b(N) \le  B,\ 
\ell \le d + b \le u,\ 
  d \ge \0,\  b\ge \0
\},
\]
i.e., $R$ is the set of vectors
$(d, b) \in \Z^n \times \Z^n$ satisfying the constraints
of the problem (DR), except for the L1-distance constraint
$\|(\bar{d} + \bar{b}) - (d+b)\|_1 \le  2\gamma$.
 We also denote
$N(d,b)
 = N_1(d,b) \cup N_2(d,b) \cup \cdots  \cup N_6(d,b)$,
where
\begin{align*}
N_1(d,b)
& =  \{(d + \chi_i - \chi_j, b)  \in \Z^n \times \Z^n
\mid  
i, j \in N,\ i \ne j\},
\notag
\\
N_2(d,b)
& =\{(d - \chi_j, b + \chi_i) \in \Z^n \times \Z^n \mid  
i, j \in N,\ i \ne j\},
\notag
\\
N_3(d,b)
& =\{(d + \chi_i, b - \chi_j) \in \Z^n \times \Z^n \mid  
i, j \in N,\ i \ne j\},
\notag
\\
N_4(d,b)
& =\{(d, b + \chi_i - \chi_j) \in \Z^n \times \Z^n \mid  
i, j \in N,\ i \ne j\},
\notag
\\
N_5(d,b)
& =\{(d - \chi_j + \chi_t, b + \chi_i - \chi_t) \in \Z^n \times \Z^n \mid  
i, j \in N,\ i \ne j,\ t \in N \setminus \{i,j\}  \},
\\
N_6(d,b)
& =
\{(d - \chi_s + \chi_i, b + \chi_s - \chi_j) \in \Z^n \times \Z^n \mid  
i, j \in N,\ i \ne j,\ s \in N \setminus \{i,j\}
 \}.
\notag
\end{align*}
\noindent
 The following property 
follows immediately from Lemma \ref{lem:sra-opt-update}.

\begin{lemma}
\label{lem:bdr-neighbor}
 For  $x \in \dom f$, and  an optimal  solution $b \in \Z^n$ of
 {\rm  (SRA$(x)$)},
we have
\begin{align*}
 \min\{f(x + \chi_i - \chi_j) \mid i,j \in N,\  
\ell \le x + \chi_i - \chi_j \le u \}
 & = 
\min\{c(d', b') \mid (d', b') \in N(d,b) \cap R \}.
\end{align*}
\end{lemma}

 By Lemma \ref{lem:bdr-neighbor}, 
the algorithm {\sc SteepestDescentDR} can be rewritten
as follows in terms of original variables $(d, b)$ as follows,
which  is nothing but the greedy algorithm by Freund et al.~\cite{FHS2017}.

\begin{flushleft}
 \textbf{Algorithm} {\sc SteepestDescentDR$'$} \\
 \textbf{Step 0:} 
 Set $d_0: = \bar{d}$, $b_0 :=\bar{b}$, and $k:=1$.
\\
 \textbf{Step 1:} 
If  $k-1 = \gamma$,
then output the solution $(d_{k-1}, b_{k-1})$ and stop.\\
 \textbf{Step 2:} 
 Find $(d', b') \in N(d_{k-1}, b_{k-1}) \cap R$
that minimizes $c(d', b')$.\\
 \textbf{Step 3:} 
 Set $(d_{k}, b_{k}) := (d', b')$ and go to Step 1.
\end{flushleft}

 For $h = 1, 2, \ldots, 6$, the value
$\min\{c(d', b') \mid (d', b') \in N_h(d_{k-1},b_{k-1}) \cap R \}$
can be computed in $\Oh(\log n)$ time by using six binary heaps
that  maintain the following six sets of numbers,
as in \cite[Section~3.1]{FHS2017}:
\begin{align*}
 & 
\{c_i(d_{k-1}(i)+1, b(i)) - c_i(d_{k-1}(i), b(i)) \mid  i \in N\}, 
\\
 & 
\{c_i(d_{k-1}(i)-1, b(i)) - c_i(d_{k-1}(i), b(i)) \mid i \in N\}, 
\\
 & 
\{c_i(d(i), b(i)+1) - c_i(d(i), b(i)) \mid i \in N\}, 
\\
 & 
\{c_i(d(i), b(i)-1) - c_i(d(i), b(i)) \mid i \in N\}, 
\\
 & 
\{c_i(d(i)+1, b(i)-1) - c_i(d(i), b(i)) \mid i \in N\}, 
\\
 & 
\{c_i(d(i)-1, b(i)+1) - c_i(d(i), b(i)) \mid i \in N\}. 
\end{align*}
 Hence, each iteration of the algorithm can be done in 
$\Oh(\log n)$ time.
 Since the initialization of the heaps requires 
$\Oh(n)$ time, we obtain  the following result:

\begin{theorem}[\cite{FHS2017}]
The algorithm {\sc SteepestDescentDR} (and also {\sc SteepestDescentDR$'$})
can be implemented so that it runs in $\Oh(n + \gamma \log n)$ time.
\end{theorem}

\subsection{Polynomial-Time Solvability of (DR)}
\label{sec:poly-drp}

 The running time of the algorithm 
{\sc SteepestDescentDR} is proportional to the  problem parameter $\gamma$
and therefore pseudo-polynomial time.
 We show that (DR) can be solved in polynomial time
by using the approach in Section~\ref{sec:poly-MML1}.

 To apply the approach in Section~\ref{sec:poly-MML1}, 
consider the minimization problem 
of the  \Mnat-convex function  $\hat f$ in \eqref{eqn:def-fhat}
under the constraint $x(N) = D+ B$,
which is equivalent to the following:
\[
 \begin{array}{l|lll}
\mbox{(DA)} & \mbox{Minimize }  & 
\displaystyle c(d, b)\\
& \mbox{subject to } & 
\displaystyle  d(N)+b(N) = D+ B,\\
& & 
\displaystyle  b(N) \le  B,\\
& & \ell \le d + b \le u,\    d, b \in \Z^n_+,\\
& &  \bar{d} + \bar{b} -\gamma \1 \le    d + b
 \le \bar{d} + \bar{b} + \gamma \1.
 \end{array}
\]

 We analyze the time complexity required to solve the problem (DA).
 Since the effective domain of the function $\hat{f}$ 
is an interval,
we can apply Theorem \ref{thm:Mnat-polymat-time}
to obtain the following time bound.

\begin{proposition}
\label{prop:DA-time}
The problem {\rm (DA)} can be solved in
$\Oh(n\sp{3}\log (\gamma/n)\log (B/n))$ time.
\end{proposition}

By using a special structure of (DA), we can prove the
following proximity theorem, which leads to a faster algorithm for (DA).
\begin{theorem}
 Let $(d,b) \in \Z^n \times \Z^n$ be a feasible solution of {\rm (DA)}
that minimizes the value $c(d,b)$ under the condition
that all components of $d$ and $b$ are even integers.
Then, there exists some optimal solution 
$(d^*,b^*) \in  \Z^n \times \Z$ of {\rm (DA)}
such that 
$\|(d^*+b^*)  - (d+b)\|_1 \le 16 n$.
\end{theorem}

\begin{theorem}
\label{thm:fast-da}
 A proximity-scaling algorithm 
finds an optimal solution of the problem {\rm (DA)}  
in $\Oh(n \log n \log ((D+B)/n))$ time.
\end{theorem}
\noindent
Details of the proximity theorem and the proximity-scaling algorithm
 are given in Appendix.

 We then analyze the time complexity for solving the 
dock re-allocation problem (DR), provided that
an optimal solution (DA) is available.
 An application of  Theorem \ref{thm:MML1-polytime} to (DR)
immediately implies  the following time bound.

\begin{proposition}
 The problem {\rm (DR)} can be solved in 
 $\Oh(n\sp{7}(\log (\gamma/n))^2 \log (B/n))$ time. 
\end{proposition}

 To obtain a better time bound for (DR), we consider a different approach.
 The discussion in Section~\ref{sec:poly-MML1} shows that if we 
have an optimal solution $(d^\bullet, b^\bullet)$ of (DA), then the problem (DR) can be
reformulated as a problem without L1-distance constraint:
\[
 \begin{array}{l|lll}
\mbox{(DR-L)} & \mbox{Minimize }  & 
\displaystyle c(d, b)\\
& \mbox{subject to } & 
\displaystyle  d(N)+b(N) = D+ B,\\
& & 
\displaystyle  b(N) \le  B,\\
& & 
  d(P)+b(P) = \bar d(P)+ \bar b(P) + \gamma,\\
& &  
  d(N \setminus P)+b(N \setminus P) = 
\bar d(N \setminus P)+ \bar b(N \setminus P) - \gamma,
\\
& & \ell \le d + b \le u,\   d, b \in \Z^n_+,\\
& &  \bar{d} + \bar{b} -\gamma \1 \le    d + b
 \le \bar{d} + \bar{b} + \gamma \1,
 \end{array}
\]
where $P \subseteq N$ is a set given as
$P = \suppp((d^\bullet + b^\bullet) - (\bar{d}+\bar{b}))$.
 To solve the problem (DR-L) efficiently, we consider the
two problems (DR-L-A$(\alpha)$) and (DR-L-B$(\alpha)$)
with parameter $\alpha$:
\begin{align*}
& \hspace*{-7mm}
   \begin{array}{l|lll}
 \mbox{(DR-L-A$(\alpha)$)} & \mbox{Minimize }  & 
 \sum_{i \in P}c_i(d(i), b(i))\\
 & \mbox{subject to } & 
 \displaystyle  b(P) \le  \alpha,\\
 & & 
  d(P)+b(P) = \bar d(P)+ \bar b(P) + \gamma,\\
 & & \ell(i) \le d(i) + b(i) \le u(i),\ d(i), b(i) \in \Z_+\ (i \in P),\\
& &  \bar{d}(i) + \bar{b}(i) -\gamma \1 \le    d(i) + b(i)
 \le \bar{d}(i) + \bar{b}(i) + \gamma \1 \  (i \in P);
 \end{array}
\end{align*}
(DR-L-B$(\alpha)$) is defined similarly to (DR-L-A$(\alpha)$),
where $P$ is replaced with $N \setminus P$
and the first constraint $b(P) \le  \alpha$
is replaced with $b(N \setminus P) \le  B -\alpha$.
 The two problems above have (almost) the same structure
as the  problem (DA), and therefore can be solved in 
$\Oh(n \log n \log ((D+B)/n))$ time by Theorem \ref{thm:fast-da}.

 We denote by $\psi_{\rm A}(\alpha)$
(resp., $\psi_{\rm B}(\alpha)$) the optimal value of 
the problem (DR-L-A$(\alpha)$) (resp., (DR-L-B$(\alpha)$)).
 Then, it is not difficult to see that
the optimal value of the problem (DR-L) is given by
$\min_{0 \le \alpha \le B} [\psi_{\rm A}(\alpha) + \psi_{\rm B}(\alpha)]$.
 The next property shows that the minimum value of
$\psi_{\rm A}(\alpha) + \psi_{\rm B}(\alpha)$ can be computed by
binary search with respect to $\alpha$.

\begin{proposition}
\label{prop:DA-convex-B}
 The values $\psi_{\rm A}(\alpha)$ and $\psi_{\rm B}(\alpha)$
are convex functions  in $\alpha \in [0, B]$.
\end{proposition}

 Since the binary search terminates in $\Oh(\log B)$ iterations
and each iteration requires 
$\Oh(n \log n \log ((D+B)/n))$ time by Theorem \ref{thm:fast-da},
we obtain the following time bound.

\begin{theorem}
 The problem  {\rm (DR)} can be solved in 
$\Oh(n \log n \log ((D+B)/n) \log B)$ time.
\end{theorem}



\newpage

\appendix

\section{Appendix: Proofs}

\makeatletter
\edef\thetheorem{\expandafter\noexpand\thesection\@thmcountersep\@thmcounter{theorem}}
\makeatother

\subsection{Proof of Theorem \ref{thm:Mnat-convex-laminar}}

 By Theorem \ref{thm:M-char-exists} (i), it suffices to show that the following condition
holds for every $x', x'' \in \dom f$ with $x' \ne x''$:
\begin{align}
& \exists i \in \suppp(x'-x''),\ \exists j \in \suppm(x'-x''):
\notag \\
& \qquad
  f(x') + f(x'') \ge f(x' - \chi_i + \chi_j) + f(x'' + \chi_i - \chi_j).
  \label{eqn:DR-Mconv4}
\end{align}

 For $x \in \dom f$, we denote
\[
 S(x)=\{ (d,b) \in \Z^n_+ \times \Z^n_+ \mid y+z=x,\ b(N) \le B \}.
\]
 Let $x', x'' \in \dom f$ be distinct vectors, and
let $(d',b') \in S(x')$  (resp., $(d'',b'') \in S(x'')$) be a pair 
of vectors such that
$f(x') = c(d',b')$ (resp.,  $f(x'') = c(d'',b'')$).
 We denote
\begin{align*}
&  N^+ = \suppp(x'-x''), \quad
 N^- = \suppm(x'-x''), \quad
 N^0 = N \setminus (N^+ \cup N^-).
\end{align*}
 In the following, we consider only the case with
$b'(N)=b''(N)=B$ since the remaining case can be proved
similarly and more easily.
 Note that this assumption and the equation $x'(N)=x''(N)$
implies $d'(N)=d''(N)$.

 We first show by using Proposition \ref{prop:2multimod}
that the condition \eqref{eqn:DR-Mconv4}  holds
if at least one of the following four conditions holds:
\begin{align*}
\mbox{(C1) } &  
N^+  \cap  \suppp(d' - d'')  \ne \emptyset, 
\ 
N^-  \cap  \suppm(d' - d'')  \ne \emptyset, 
\\
\mbox{(C2) } &  
N^+  \cap  \suppp(b' - b'')  \ne \emptyset, 
\ 
N^-  \cap  \suppm(b' - b'')  \ne \emptyset, 
\\
\mbox{(C3) } &  
N^+  \cap  \suppp(d' - d'')  \ne \emptyset, 
\ 
N^-  \cap  \suppm(b' - b'')  \ne \emptyset, 
\ 
N^0  \cap  \suppm(d' - d'')  \ne \emptyset, 
\\
\mbox{(C4) } &  
N^+  \cap  \suppp(b' - b'')  \ne \emptyset, 
\ 
N^-  \cap  \suppm(d' - d'')  \ne \emptyset, 
\ 
N^0  \cap  \suppm(b' - b'')  \ne \emptyset. 
\end{align*}

 In the following, we give a proof for only the case with (C3);
the proof for other cases are similar and omitted.
 Let $i,j, s \in N$ be distinct elements such that
\[
i \in  N^+  \cap  \suppp(d' - d''), \quad
j \in N^-  \cap  \suppm(b' - b''), \quad
s \in N^0  \cap  \suppm(d' - d'').
\]
 Note that the choice of $s$ implies $s \in \suppp(b' - b'')$.
 We define vectors 
$\tilde d', \tilde d'', \tilde b', \tilde b'', \tilde{x}', \tilde{x}'' \in \Z^n$ by
\begin{align*}
&  \tilde d' = d' - \chi_i + \chi_s,\ \tilde d'' = d'' + \chi_i - \chi_s,\ 
 \tilde b' = b' + \chi_j - \chi_s,\  \tilde b'' = b'' - \chi_j + \chi_s,\ 
\\
& \tilde x' =  \tilde d' + \tilde b'\ (=  x'  - \chi_i + \chi_j), \  
 \tilde x'' = \tilde d'' + \tilde b'' ( =x''  + \chi_i - \chi_j).
\end{align*}
 It is not difficult to see that
$\tilde x', \tilde x'' \in \dom f$,
$(\tilde d', \tilde b') \in S(\tilde x')$,
and $(\tilde d'', \tilde b'') \in S(\tilde x'')$ hold.
 Hence, we have
\begin{equation}
   \label{eqn:DR-Mconv5}
 f(\tilde x') \le c(\tilde d', \tilde b'), \qquad 
 f(\tilde x'') \le c(\tilde d'', \tilde b'').
\end{equation}
By the choice of $i, j, s \in N$, the following inequalities follow
from Proposition \ref{prop:2multimod}:
\begin{align}
& 
c_i(d'(i), {b}'(i)) + c_i(d''(i), {b}''(i))  
\notag\\
& \ge
c_i(d'(i)-1, {b}'(i)) + c_i(d''(i)+1, {b}''(i))  
= c_i(\tilde d'(i), \tilde {b}'(i)) + c_i(\tilde d''(i),\tilde  {b}''(i)),  
  \label{eqn:DR-Mconv1}
\\
& 
c_j(d'(j), {b}'(j)) + c_j(d''(j), {b}''(j))  
\notag\\
& \ge
c_j(d'(j), {b}'(j)+1) + c_j(d''(j), {b}''(j)-1)  
= c_j(\tilde d'(j), \tilde {b}'(j)) + c_j(\tilde d''(j),\tilde  {b}''(j)),  
  \label{eqn:DR-Mconv2}
\\
& 
c_s(d'(s), {b}'(s)) + c_s(d''(s), {b}''(s))  
\notag\\
& \ge
c_s(d'(s)+1, {b}'(s)-1) + c_s(d''(s)-1, {b}''(s)+1)  
= c_s(\tilde d'(s), \tilde {b}'(s)) + c_s(\tilde d''(s),\tilde  {b}''(s)).  
  \label{eqn:DR-Mconv3}
\end{align}
 From these inequalities and \eqref{eqn:DR-Mconv5} follows that
\begin{align*}
  f(x') + f(x'') 
& = c( d',  b') + c( d'',  b'') \\
& \ge  c(\tilde d', \tilde b') + 
 c(\tilde d'', \tilde b'')\\
& \ge f(\tilde x') + f(\tilde x'') 
= f(x' - \chi_i + \chi_j) + f(x'' + \chi_i - \chi_j). 
\end{align*}
 This shows that the inequality \eqref{eqn:DR-Mconv4} holds.

 To conclude the proof, we show that
at least one of the four conditions (C1)--(C4) holds.
 Assume, to the contrary, that neither of the four conditions
holds.
 Since $x'(N)=x''(N)$ and $x' \ne x''$,
we have $N^+ \ne \emptyset$,
which implies at least one of 
$N^+ \cap \suppp(d'-d'') \ne \emptyset$ 
and $N^+ \cap \suppp(b'-b'') \ne \emptyset$ holds;
we may assume that the former holds.
 Since (C1)  does not hold, we have
$N^-  \cap  \suppm(d' - d'') = \emptyset$,
which implies that $N^-  \subseteq  \suppm(b' - b'')$.
 Since (C2)  does not hold, we have
$N^+  \cap  \suppp(b' - b'') = \emptyset$,
which implies that $N^+  \subseteq  \suppp(d' - d'')$.
 Since (C3)  does not hold, we have
$N^0  \cap  \suppm(d' - d'')  = \emptyset$.
 Hence, we have 
\[
 d'(N^+) > d''(N^+), \ d'(N^-) \ge d''(N^-), \ 
d'(N^0) \ge d''(N^0),
\]
implying that $d'(N) > d''(N)$, a contradiction
to the equation $d'(N)=d''(N)$. 
 This concludes the proof.

\subsection{Proof of Theorem \ref{thm:norm-opt-main}}
\label{sec:proof-Th3.3}

 We prove  Theorem \ref{thm:norm-opt-main} in this section.
 For this, we show some technical lemmas.

\begin{lemma}
\label{lem:suppp-suppm} 
 Let $y, \tilde{y} \in \Z^n$ be distinct vectors satisfying $y(N)=\tilde{y}(N)$.
 If $\|y - \mycenter\|_1 \le \|\tilde{y} -\mycenter\|_1$,
then we have 
 $\tilde{y}(i) > \mycenter(i)$ for some $i \in \suppp(\tilde{y} - y)$ 
 or $\tilde{y}(j) < \mycenter(j)$ for some $j \in \suppm(\tilde{y} - y)$  (or both).
\end{lemma}

\begin{proof}
 We prove the statement by contradiction.
 Assume, to the contrary, that
 $\tilde{y}(i) \le \mycenter(i)$ for all $i \in \suppp(\tilde{y} - y)$ 
and  $\tilde{y}(j) \ge \mycenter(j)$ for all $j \in \suppm(\tilde{y} - y)$.
 Then, it holds that
 \begin{align*}
 \|\tilde{y}-\mycenter\|_1 - \|y-\mycenter\|_1 
 & = \sum_{i \in \suppp(\tilde{y} - y)} (|\tilde{y}(i)-\mycenter(i)| - |y(i)-\mycenter(i)|)\\
 & \quad +\sum_{j \in  \suppm(\tilde{y} - y)} (|\tilde{y}(j)-\mycenter(j)| -
  |y(j)-\mycenter(j)|)\\
 & = \sum_{i \in \suppp(\tilde{y} - y)} 
 [(\mycenter(i) - \tilde{y}(i) ) - (\mycenter(i) - y(i))]\\
 & \quad +\sum_{j \in  \suppm(\tilde{y} - y)} [(\tilde{y}(j)-\mycenter(j)) -
  (y(j)-\mycenter(j)|)]\\
 & = 
\sum_{i \in \suppp(\tilde{y} - y)} 
 (- \tilde{y}(i) + y(i))
+\sum_{j \in \suppm(\tilde{y} - y)} (\tilde{y}(j)-y(j)) < 0,
 \end{align*}
a contradiction to the inequality $\|y - \mycenter\|_1 \le \|\tilde{y} -\mycenter\|_1$.
\end{proof}

\begin{lemma}
\label{lem:norm-sepconv}
 Let $x, y, z \in \Z^n$, $i \in \suppp(x-y)$, and $j \in\suppm(x-y)$.
 Then, we have
\[
\|x-z\|_1 + \|y - z\|_1
\ge \|(x  - \chi_i + \chi_j) -z \|_1 + \|(y + \chi_i - \chi_j) - z\|_1.
\]
\end{lemma}

\begin{proof}
 For a univariate convex function $\phi: \R \to \R$
and integers $\eta, \zeta$ with $\eta < \zeta$, it holds that
\[
 \phi(\eta) +  \phi(\zeta) \ge
 \phi(\eta+1) +  \phi(\zeta-1).
\]
 We have $\|x-z\|_1 = \sum_{i=1}^n |x_i - z_i|$
and each term $|x_i - z_i|$ is a univariate convex function
in $x_i$, the claim follows.
\end{proof}

 We say that a sequence $y_0, y_1, \ldots, y_h \in \dom f$ 
of vectors is \textit{monotone}
if $\|y_k - y_0\|_1 = 2k$ holds for $k=0,1,\ldots, h$. 
 This condition can be rewritten as follows:
\[
\begin{array}{ll}
 \mbox{for }k=0,1,\ldots, h-1,\ \mbox{ it holds that }
 y_{k+1} = y_k - \chi_i + \chi_j \\
 \mbox{ for some }
 i \in \suppp(y_k - y_h) \mbox{ and }j \in \suppm(y_k - y_h).
\end{array}
\]
 Recall that by the definition of $\tau$, 
every optimal solution of the problem {\rm (MML1$(\tau)$)} 
is a minimizer of $f$.

\begin{lemma}
\label{lem:1-5}
 Let $y \in \dom f$ be a vector with $\|y - \mycenter\|_1 <2\tau$,
 and $x^\bullet \in M_\tau$ be a vector minimizing the value $\|x^\bullet - y\|_1$.
 Then, there exists a monotone sequence $y_0, y_1, \ldots, y_h \in \dom f$ 
with $h = (1/2)\|y - x^\bullet\|_1$ such that
$y_0 = y$, $y_h = x^\bullet$, and 
$f(y_0) > f(y_1) > \cdots > f(y_h)$.
\end{lemma}

\begin{proof}
 We prove the claim by induction on $h$.
 It suffices to show that
there exists some $i \in \suppp(y-x^\bullet)$ and $j \in \suppm(y-x^\bullet)$
such that $f(y - \chi_i +  \chi_j) < f(y)$
since $(1/2)\|(y - \chi_i +  \chi_j) - x^\bullet\|_1 = h-1$.

 Since $\|x^\bullet - \mycenter\|_1 = 2 \tau > \|y - \mycenter\|_1$, 
it follows from Lemma \ref{lem:suppp-suppm} that
 $x^\bullet(i) > \mycenter(i)$ for some $i \in \suppp(x^\bullet - y)$ 
 or $x^\bullet(j) < \mycenter(j)$ for some $j \in \suppm(x^\bullet - y)$  (or both);
we assume, without loss of generality, that the former holds.
 Then, the exchange property (M-EXC) of M-convex function $f$
applied to  
$x^\bullet, y$, and $i$ implies that there exists some
 $j \in \suppm(x^\bullet-y)$ such that
\begin{equation}
\label{eqn:lem1-5:1}
  f(x^\bullet) + f(y) \ge f(x^\bullet - \chi_i + \chi_j) + f(y + \chi_i - \chi_j).
\end{equation}
 Hence, if we have $f(x^\bullet) < f(x^\bullet - \chi_i + \chi_j)$,
then  (\ref{eqn:lem1-5:1}) implies 
the desired inequality $f(y - \chi_i +  \chi_j) < f(y)$.
 In the following, we prove $f(x^\bullet) < f(x^\bullet - \chi_i + \chi_j)$.

 By the choice of $i$, we have
$\|(x^\bullet - \chi_i + \chi_j) - \mycenter\|_1 - \|x^\bullet  - \mycenter\|_1 \in \{0, - 2\}$.
If
$\|(x^\bullet - \chi_i + \chi_j) - \mycenter\|_1 - \|x^\bullet  - \mycenter\|_1 =0$
then we have 
$f(x^\bullet) < f(x^\bullet - \chi_i + \chi_j)$
by the choice of $x^\bullet$ since $\|(x^\bullet - \chi_i + \chi_j) - y\|_1 < 
\|x^\bullet  - y\|_1$.
If $\|(x^\bullet - \chi_i + \chi_j) - \mycenter\|_1 - \|x^\bullet  - \mycenter\|_1 =-2$
then we have $\|(x^\bullet - \chi_i + \chi_j) - \mycenter\|_1 < 2\tau$
and therefore 
$f(x^\bullet) < f(x^\bullet - \chi_i + \chi_j)$ holds by the definition of $\tau$.
 Hence, we have $f(x^\bullet) < f(x^\bullet - \chi_i + \chi_j)$ in either case.
\end{proof}

 We now prove the claims (i), (ii), and (iii) of 
Theorem \ref{thm:norm-opt-main}  in turn.

\begin{proof}[Proof of Theorem \ref{thm:norm-opt-main} (i)]
 We first show that $\mu_k > \mu_{k+1}$ for each integer $k \in [\sigma, \tau-1]$.
 Let $y \in M_k$, and 
 $x^\bullet \in M_\tau$ be a vector that minimizes the value $\|x^\bullet - y\|_1$.
 Note that $x^\bullet \in \arg\min f$,
and by the induction hypothesis we have $\|y - \mycenter\|_1 = 2k$.
 By Lemma \ref{lem:1-5},
there exists a monotone sequence $y_0, y_1, \ldots, y_h \in \dom f$ 
with $h=\|x^\bullet - y\|_1$ such that $y_0 = y$, $y_h = x^\bullet$, and 
$\mu_k = f(y_0) > f(y_1) > \cdots > f(y_h)$.
 Since $\|y_{t+1} - \mycenter\|_1 - \|y_t - \mycenter\|_1 \in \{-2, 0,+2\}$
for every integer $t \in [0,h-1]$
and $\|y_h - \mycenter\|_1 = 2\tau > 2k \ge \|y_0 - \mycenter\|_1$,
there exists some integer $s \in [1, h]$ such that
$\|y_s - \mycenter\|_1 = 2(k+1)$;
such $s$ satisfies 
$\mu_{k+1} \le f(y_s) < f(y_0) = \mu_k$.

 The inclusion
$M_k \subseteq \{x \in \Z^n \mid \|x - \mycenter\|_1 = 2k \}$
follows from the inequality $\mu_{k} < \mu_{k-1}$ 
since $f(x) \ge \mu_{k-1} > \mu_k$ holds for every $x \in \dom f$
with $\|x- \mycenter\|_1 < 2k$.
\end{proof}

\begin{proof}[Proof of Theorem \ref{thm:norm-opt-main} (ii)]
 We fix $y \in M_k$, and 
 let $\tilde{y}$ be a vector in $M_{k+1}$ that minimizes $\|\tilde{y} - y\|_1$.
 By Lemma \ref{lem:suppp-suppm},  
it suffices to consider the following two cases:
\begin{quote}
 Case 1: 
$\suppp(\tilde{y} - y) \cap \suppp(\tilde{y} - \mycenter) \ne \emptyset$,
\\
 Case 2: 
$\suppm(\tilde{y} - y) \cap \suppm(\tilde{y} - \mycenter) \ne \emptyset$.
\end{quote}
 In the following we give a proof for Case 1 only since Case 2
can be proven in a similar way.

 Suppose that there exists some $i \in 
\suppp(\tilde{y} - y) \cap \suppp(\tilde{y} - \mycenter)$.
 By (M-EXC) applied to $\tilde{y}$ and $y$, there exists some
 $j \in \suppm(\tilde{y}-y)$ such that
 \begin{equation}
 \label{eqn:opt-adj-1}
  f(\tilde{y}) + f(y) \ge f(\tilde{y} - \chi_i + \chi_j) + f(y + \chi_i - \chi_j).
 \end{equation}
 Put $\tilde{z} = \tilde{y} - \chi_i + \chi_j$, $z= y + \chi_i - \chi_j$,
and 
\[
 \alpha = \|\tilde{z}-\mycenter\|_1 - \|\tilde{y}-\mycenter\|_1, \quad
\beta = \|z-\mycenter\|_1 - \|y-\mycenter\|_1.
\]
 Then, we have $\beta \in \{-2, 0, +2\}$
and $\alpha \in \{-2, 0\}$ since  $\tilde{y}(i) > \mycenter(i)$.

 Assume first that $\alpha= 0$ holds.
 By Lemma \ref{lem:norm-sepconv}, we have
\[
\alpha + \beta 
= \|\tilde{z}-\mycenter\|_1 + \|z-\mycenter\|_1
-  \|\tilde{y}-\mycenter\|_1 - \|y-\mycenter\|_1
\le 0.
\]
 This, together with $\alpha = 0$, implies $\beta \le 0$.
 Hence, it holds that $\|z -\mycenter\|_1 \le  \|y-\mycenter\|_1 = 2k$,
implying $f(z) \ge \mu_{\|z -\mycenter\|_1/2} \ge \mu_k$
by Theorem \ref{thm:norm-opt-main} (i).
  Since $\|\tilde{z}-\mycenter\|_1 = \|\tilde{y}-\mycenter\|_1 = 2(k+1)$,
we have $f(\tilde{z}) \ge \mu_{k+1}$.
 Combining these inequalities with  (\ref{eqn:opt-adj-1}),   we have
 \[
 \mu_{k+1}+\mu_k = 
  f(\tilde{y}) + f(y) \ge f(\tilde{z}) + f(z)
 \ge \mu_{k+1}+\mu_{k},
 \]
 from which follows that $f(\tilde{z}) = \mu_{k+1}$,
 a contradiction to the choice of $\tilde{y}$
 since $\|\tilde{z} - y\|_1 = \|\tilde{y} - y\|_1 -2$.
 This shows that $\alpha = 0$ cannot occur.
 Hence, we have $\alpha = -2$.

  Since $\alpha = -2$, 
we have $\|\tilde{z}-\mycenter\|_1 = \|\tilde{y}-\mycenter\|_1 - 2 = 2k$,
from which follows that $f(\tilde{z}) \ge \mu_{k}$.
 We also have
 $\|z -\mycenter\|_1 \le  \|y-\mycenter\|_1 + 2= 2(k+1)$,
 and therefore
 $f(z) \ge \mu_{\|z -\mycenter\|_1/2} \ge \mu_{k+1}$,
where the last inequality is 
by Theorem \ref{thm:norm-opt-main} (i).
 Combining these inequalities with  (\ref{eqn:opt-adj-1}), 
 we have
 \[
 \mu_{k+1}+\mu_k = 
  f(\tilde{y}) + f(y) \ge f(\tilde{z}) + f(z)
 \ge \mu_{k}+\mu_{k+1},
 \]
 from which follows that $f(z) = \mu_{k+1}$.
 This implies $\|z - \mycenter\|_1 = 2(k+1)$
since $\mu_{k-1} > \mu_k > \mu_{k+1}$ by 
Theorem \ref{thm:norm-opt-main} (i).
 Hence, we have $z = y + \chi_i - \chi_j \in M_{k+1}$,
$i \in N \setminus \suppm(y - \mycenter)$,
and $j \in N \setminus \suppp(y - \mycenter)$.
\end{proof}

\begin{proof}[Proof of Theorem \ref{thm:norm-opt-main}  (iii)]
 The proof below is quite similar to that for 
Theorem \ref{thm:norm-opt-main}  (ii).

 We fix $y' \in M_{k+1}$, and 
 let $y$ be a vector in $M_{k}$ that minimizes $\|y - y'\|_1$.
 By Lemma \ref{lem:suppp-suppm},  
it suffices to consider the following two cases:
\begin{quote}
 Case 1: 
$\suppp({y}' - y) \cap \suppp({y}' - \mycenter) \ne \emptyset$,
\\
 Case 2: 
$\suppm({y}' - y) \cap \suppm({y}' - \mycenter) \ne \emptyset$.
\end{quote}
 In the following we give a proof for Case 1 only since Case 2
can be proven in a similar way.

 Suppose that there exists some $i \in 
\suppp({y}' - y) \cap \suppp({y}' - \mycenter) \ne \emptyset$.
 By (M-EXC) applied to $y'$ and $y$, there exists some
$j \in \suppm(y'-y)$ such that
\begin{equation}
 \label{eqn:opt-adj-2}
  f(y') + f(y) \ge f(y' - \chi_i + \chi_j) + f(y + \chi_i - \chi_j).
\end{equation}
 Put $z' = y' - \chi_i + \chi_j$, $z= y + \chi_i - \chi_j$,
and 
\[
 \alpha = \|z'-\mycenter\|_1 - \|y'-\mycenter\|_1, \quad
\beta = \|z-\mycenter\|_1 - \|y-\mycenter\|_1.
\]
 Then, we have $\beta \in \{-2, 0, +2\}$
and $\alpha \in \{-2, 0\}$ since  ${y}'(i) > \mycenter(i)$.

 Assume first that $\alpha= 0$ holds.
 By Lemma \ref{lem:norm-sepconv}, we have
\[
\alpha + \beta 
= \|{z}'-\mycenter\|_1 + \|z-\mycenter\|_1
-  \|{y}'-\mycenter\|_1 - \|y-\mycenter\|_1
\le 0.
\]
 This, together with $\alpha = 0$, implies $\beta \le 0$.
 Hence, it holds that $\|z -\mycenter\|_1 \le  \|y-\mycenter\|_1 = 2k$,
 and therefore  $f(z) \ge \mu_{\|z -\mycenter\|_1/2} \ge \mu_k$
by Theorem \ref{thm:norm-opt-main} (i).
  Since $\|z'-\mycenter\|_1 = \|y'-\mycenter\|_1 = 2(k+1)$,
we have $f(z') \ge \mu_{k+1}$.
 Combining these inequalities with  (\ref{eqn:opt-adj-2}),   we have
 \[
 \mu_{k+1}+\mu_k = 
  f(y') + f(y) \ge f(z') + f(z)
 \ge \mu_{k+1}+\mu_{k},
 \]
 from which follows that $f(z') = \mu_{k+1}$,
 a contradiction to the choice of $y'$
 since $\|z' - y\|_1 = \|y' - y\|_1 -2$.
 This shows that $\alpha = 0$ cannot occur.
 Hence, we have $\alpha = -2$.

  Since $\alpha = -2$, 
we have $\|z'-\mycenter\|_1 = \|y'-\mycenter\|_1 - 2 = 2k$,
implying that $f(z') \ge \mu_{k}$.
 We also have
 $\|z -\mycenter\|_1 \le  \|y-\mycenter\|_1 + 2= 2(k+1)$,
 and therefore
 $f(z) \ge \mu_{\|z -\mycenter\|_1/2} \ge \mu_{k+1}$,
where the last inequality is 
by Theorem \ref{thm:norm-opt-main} (i).
 Combining these inequalities with  (\ref{eqn:opt-adj-2}), 
 we have
 \[
 \mu_{k+1}+\mu_k = 
  f(y') + f(y) \ge f(z') + f(z)
 \ge \mu_{k}+\mu_{k+1},
 \]
 from which follows that $f(z') = \mu_{k}$.
 Hence, we have $z' = y' - \chi_i + \chi_j \in M_{k}$,
$i \in \suppp(y - \mycenter)$
and $j \in \suppm(y - \mycenter)$.
\end{proof}

\subsection{Reverse Steepest Descent Algorithm for (MML1)}

 We can also consider another variant of steepest descent algorithm
that starts from a nearest minimizer $x^\bullet$ of $f$
and greedily approaches $\mycenter$.
 This algorithm finds an optimal solution of (MML1) faster than
{\sc SteepestDescentMML1}
if $\tau - \gamma$ is smaller than $\gamma-\sigma$.

\begin{flushleft}
 \textbf{Algorithm} {\sc ReverseSteepestDescentMML1} 
\\
 \textbf{Step 0:} 
 Compute the value $\tau$ in \eqref{eqn:def-sigma}
and a minimizer  $\xbullet$ of $f$ with 
$\|\xbullet - \mycenter \|_1 = 2\tau$.
\\
\noindent
\phantom{Step 0: }
Set $x_\tau: = x^\bullet$, and $k:=\tau-1$.
\\
 \textbf{Step 1:} 
If $k+1 = \gamma$,
then output $x_{k+1}$ and stop.\\
 \textbf{Step 2:} 
 Find $i_k, j_k \in N$ that minimizes $f(x_{k+1} - \chi_{i_k} + \chi_{j_k})$.\\
 \textbf{Step 2:} 
 Set $x_{k} := x_{k+1} - \chi_{i_k} + \chi_{j_k}$, 
$k:=k-1$, and go to Step 1.
\end{flushleft}

\begin{theorem}
\label{thm:main-reverse}
The algorithm {\sc ReverseSteepestDescentMML1} applied to
an M-convex function $f: \Z^n \to \Rinf$
outputs an optimal solution of {\rm (MML1)}
in $\tau- \gamma$ iterations.
\end{theorem}

\begin{proof}
 In a similar way as in the proof of Theorem \ref{thm:main1},
we can show that $x_k \in M_k$ holds
for $k = \tau, \tau-1, \ldots, \gamma$.
 Hence, the output $x_\gamma$ of the algorithm is an optimal solution of (MML1).
\end{proof}

\subsection{Remarks in Section \ref{sec:sda-MML1}}

 As an immediate corollary of Theorem \ref{thm:main1},
we can obtain the following property of the algorithm {\sc SteepestDescent} 
in the case with $\mycenter \in \dom f$.
 Note that the behavior of the algorithm {\sc SteepestDescentMML1} 
coincides with that of {\sc SteepestDescent} 
if $\mycenter \in \dom f$ and $\gamma \ge \tau$.

\begin{corollary}
\label{coro:M-steep-descent}
 Let $f: \Z^n \to \Rinf$ be an M-convex function with 
$\mycenter \in \dom f$.
 Suppose that the algorithm {\sc SteepestDescent} is applied to $f$  
with $\mycenter$  as an initial vector.
 Then, the algorithm terminates in exactly $\tau$ iterations
and outputs an optimal solution of {\rm (MML1)}.
\end{corollary}


\begin{remark}\rm
 In Corollary \ref{coro:M-steep-descent}
we obtained the exact bound
on the number of iterations required by the algorithm {\sc SteepestDescent}.
 While this bound is shown 
for some special case of M-convex functions 
and for some  variants of the algorithm,
it is not proven so far for 
the ``naive'' steepest descent algorithm (i.e., {\sc SteepestDescent}).

 The same bound for {\sc SteepestDescent} is obtained by \cite{Murota03} 
for the special case where an M-convex function has a unique minimizer.
 Based on this fact, 
the same bound for general M-convex functions is obtained in \cite{Murota03},
where a variant of {\sc SteepestDescent}
with certain tie-breaking rules 
 in the choice of $i_k$ and $j_k$ in Step 1 is used.
 The same  bound can be also obtained 
by using another variant of {\sc SteepestDescent} in \cite{Shioura04},
where a region containing an optimal solution
is explicitly maintained by lower and upper bound vectors.
 Corollary \ref{coro:M-steep-descent} shows that
no modification of the algorithm {\sc SteepestDescent} 
is necessary to obtain the same exact bound.
\qed
\end{remark}

\begin{remark}\rm
 It can be shown that the sequence of optimal values $\mu_k$ for (MML1$(k)$)
is a convex sequence.

\begin{theorem}
\label{thm:convex-seq}
For every integer $k \in [1,\tau-1]$,
it holds that $\mu_{k-1} + \mu_{k+1} \ge 2\mu_k$. 
\end{theorem}
\qed
\end{remark}

\begin{proof}[Proof of Theorem \ref{thm:convex-seq}]
For $k=1,2,\ldots, \tau-1$, let $x_{k-1} \in M_{k-1}$
and  $x_{k+1} \in M_{k+1}$ be vectors such that
\[
 x_{k+1} = x_{k-1} - \chi_i - \chi_{i'} + \chi_{j} + \chi_{j'}
\]
for some $i, i', j, j' \in  N$ with 
$\{i, i'\} \cap \{j, j'\} = \emptyset$,
$i, i' \in \suppm(x_{k+1} - \mycenter)$,
and
$j, j' \in \suppp(x_{k+1} - \mycenter)$;
the existence of such $x_{k-1}$ and $x_{k+1}$
follows from the claim (ii) (or (iii)) of Theorem \ref{thm:norm-opt-main}.
 By (M-EXC) applied to $x_{k-1}$ and $x_{k+1}$,
we have 
$f(x_{k-1}) + f(x_{k+1}) \ge  f(y) + f(z)$
with $(y,z) = (x_{k-1}- \chi_i + \chi_{j}, x_{k-1}- \chi_{i'} + \chi_{j'})$
or $(y,z) = (x_{k-1}- \chi_i + \chi_{j'}, x_{k-1}- \chi_{i'} + \chi_{j})$.
In either case we have $\|y - \mycenter\|_1 = \|z - \mycenter\|_1 = 2k$,
and therefore follows that
\[
\mu_{k-1} + \mu_{k+1} = 
f(x_{k-1}) + f(x_{k+1}) \ge  f(y) + f(z) \ge 2 \mu_k.
\]
\end{proof}

\subsection{Proof of Lemma \ref{lem:MML-MML1}}

 We first show that every feasible solution $x$ of (MM-L) satisfies
the L1-distance constraint $\|x - \mycenter\|_1 \le 2\gamma$.
 Under the condition $\hat{\ell} \le x \le \hat u$ we have
\[
 \|x - \mycenter\|_1 = (x(P) - \mycenter(P)) + 
(\mycenter(N \setminus P) - x(N \setminus P)),
\]
and the equation $x(N) = \theta = \mycenter(N)$
implies that
$x(P) - \mycenter(P) = \mycenter(N \setminus P) - x(N \setminus P)$.
 Since $x(P) = \mycenter(P) + \gamma$,
the L1-distance $\|x - \mycenter\|_1$ is bounded by $2\gamma$.
 
 To conclude the proof, it suffices to 
show that there exists an optimal solution $x^*$ of (MML1)
such that $x^*(P) = \mycenter(P) + \gamma$
and $\hat{\ell} \le x^* \le \hat u$.
 Repeated use of Theorem \ref{thm:norm-opt-main} (iii) implies that
there exists an optimal solution  $x^* \in \dom f$ of
{\rm (MML1)} such that $\|x^* - \mycenter\|_1 = 2\gamma$,
$x^*(P) = \mycenter(P) + \gamma$, 
$x^*(N \setminus P) = \mycenter(N \setminus P) - \gamma$, and
\begin{equation}
\label{eqn:optsol-ulbounds}
 \mycenter(i) \le x^*(i) \le \xbullet(i)\ (i \in P),
\quad
 \xbullet(i) \le x^*(i) \le \mycenter(i) \ (i \in N \setminus P).
\end{equation}
 By the equation $x^*(P) = \mycenter(P) + \gamma$, 
for $i \in P$ the upper bound of $x^*(i)$ in \eqref{eqn:optsol-ulbounds}
can be replaced with
$\min\{\xbullet(i), \mycenter(i) + \gamma\}$.
 Similarly, 
for $i \in N \setminus P$ the lower bound of $x^*(i)$ in \eqref{eqn:optsol-ulbounds}
can be replaced with
$\max\{\xbullet(i), \mycenter(i) - \gamma\}$.
 This concludes the proof.

\subsection{Proof of Proposition \ref{prop:g-Mnat-polymat}}

 We show that $g$ is M-convex.
 Define a function $f': \Z^n \to \Rinf$  by
\[
 f'(x) = 
\begin{cases}
f(x)  & (\mbox{if } x(N) = \theta,\ 
\hat{\ell}(i) \le x(i) \le  \hat{u}(i)\ (i \in P)),\\
+ \infty & (\mbox{otherwise}).
\end{cases}
\]
 Then, $f'$ is an M-convex function \cite[Theorem~6.13 (5)]{Murota03book}.
 The function 
$g': \Z^{N\setminus P} \to \Rinf$ given by
\[
 g'(y) = \min\{f'(x) \mid  x(i)=y(i)\ (i \in N \setminus P) \}
\qquad (y \in \Z^{N \setminus P})
\]
is an  \Mnat-convex function since $g'$ is a projection
of the M-convex function $f'$ \cite[Theorem~6.15~(2)]{Murota03book}.
 Finally, function $g$ is given as
\[
 g(y) =
\begin{cases}
 g'(y) & 
 (\mbox{if } 
y(N \setminus P) = \theta - (\mycenter(P) + \gamma),\ 
\hat{\ell}(i) \le x(i) \le  \hat{u}(i)\ (i \in N \setminus P)),\\
+ \infty & (\mbox{otherwise}),
\end{cases}
\]
and therefore $g$ is M-convex (cf.~\cite[Theorem~6.13]{Murota03book}).
%

\subsection{Proof of Lemma  \ref{lem:sra-opt-update}}

 We denote $\hat{x} = x + \chi_i - \chi_j$, and
let $\hat{b} \in \Z^n$ be
 an optimal  solution  of the problem {\rm  (SRA$(\hat x)$)} 
that  minimizes the value $\|\hat{b} - {b}\|_1$.
 We show that the vector $\hat{b}$ satisfies
the condition~\eqref{eqn:sra-opt-update-cond1}.

\medskip

\noindent 
Claim 1:
 If $s \in N$ satisfies either
$s \in \suppp(\hat{b}- {b}) \setminus \{i\}$
or $s = i$ and $\hat{b}(i)- b(i) \ge 2$, 
then
\begin{align}
& 
c_s(\hat x(s) -\hat{b}(s), \hat{b}(s)) + c_s(x(s)- {b}(s), {b}(s))
\notag\\
& \ge
 c_s(\hat x(s) - \hat{b}(s)+1, \hat{b}(s)-1)
+ c_s(x(s) - {b}(s) - 1, {b}(s)+1).
  \label{eqn:sra-opt-update-7}
\end{align}
 If $t \in N$ satisfies either
$t \in \suppm(\hat{b}- {b}) \setminus \{j\}$
or $t = j$ and $\hat{b}(j)- b(j) \le -2$, 
then
\begin{align}
& 
c_t(\hat x(t) -\hat{b}(t), \hat{b}(t))
+ c_t(x(t)- {b}(t), {b}(t))
\notag\\
& \ge
 c_t(\hat x(t) - \hat{b}(t)-1, \hat{b}(t)+1)
+ c_t(x(t) - {b}(t)+1, {b}(t)-1).
  \label{eqn:sra-opt-update-8}
\end{align}

\noindent
[Proof of Claim] \quad
 We prove the inequality \eqref{eqn:sra-opt-update-7} only
since \eqref{eqn:sra-opt-update-8} can be shown similarly.
If $s \in \suppp(\hat{b}- {b}) \setminus \{i\}$ then
we have $\hat{b}(s) > {b}(s)$
and $\hat{x}(s) - \hat{b}(s) < x(s)- b(s)$.
 If $s=i$ and 
$\hat{b}(i)- b(i) \ge 2$, then we have
$\hat x(i) -\hat{b}(i) \le x(i)+1 - (b(i)+2) < x(i) - b(i)$.
In either case, 
\eqref{eqn:sra-opt-update-7} follows
by Proposition \ref{prop:2multimod} (i).
[End of Claim]

\medskip

 To prove the lemma, we consider the following two conditions:
\begin{align*}
\mbox{(a) }  &
\mbox{$\suppp(\hat{b}- {b}) \subseteq \{i\}$
and  $\hat{b}(i)- b(i) \le 1$,} 
\\
\mbox{(b) }  &
\mbox{$\suppm(\hat{b}- {b}) \subseteq\{j\}$
and $\hat{b}(j)- b(j) \ge -1$.} 
\end{align*}

\medskip

\noindent
Claim 2: 
 At least one of the conditions (a) and (b) holds.
 Moreover, the condition (a) holds if $b(N) < B$,
and the condition (b) holds if $\hat{b}< B$.

\noindent
{[Proof of Claim]} \quad
 To prove the former statement,
assume, to the contrary, that
there exist some $s, t \in N$
satisfying the following two conditions:
\begin{align*}
  &
\mbox{$s \in \suppp(\hat{b}- {b}) \setminus \{i\}$
or $s=i$ and $\hat{b}(i)- b(i) \ge 2$,} 
\\
  &
\mbox{$t \in \suppm(\hat{b}- {b}) \setminus \{j\}$
or $t=j$ and $\hat{b}(j)- b(j) \le -2$.} 
\end{align*}
 Then, it holds that
\begin{align}
& 
c(\hat x - \hat b,  \hat b) + c(x - b,   b) 
\notag\\
&  \quad
-   c(\hat x - (\hat b - \chi_s + \chi_t),  \hat b - \chi_s + \chi_t)
- c(x - (b +\chi_s - \chi_t),  b + \chi_s - \chi_t) 
\notag\\
& =
\bigg[ 
c_s(\hat x(s) -\hat{b}(s), \hat{b}(s)) + c_s(x(s)- {b}(s), {b}(s))
\notag\\
& \qquad
-
 c_s(\hat x(s) - \hat{b}(s)+1, \hat{b}(s)-1)
- c_s(x(s) - {b}(s) - 1, {b}(s)+1)
\bigg]
\notag\\
& \quad +
\bigg[
c_t(\hat x(t) -\hat{b}(t), \hat{b}(t)) + c_t(x(t)- {b}(t), {b}(t))
\notag\\
& \qquad
- c_t(\hat x(t) - \hat{b}(t)-1, \hat{b}(t)+1)
- c_t(x(t) - {b}(t)+1, {b}(t)-1)
\bigg]
\notag\\
& \ge 0,
  \label{eqn:sra-opt-update-9}
\end{align}
where the inequality is by
 \eqref{eqn:sra-opt-update-7} and \eqref{eqn:sra-opt-update-8}
in Claim 1.
 Note that ${b}+ \chi_s - \chi_t$ is a feasible solution of (SRA($x$))
since $({b}+ \chi_s - \chi_t)(N)=b(N) \le B$,
$b(s) < \hat{b}(s) \le x(s)$, and $b(t) > \hat{b}(t) \ge 0$.
 Therefore,  we have
\[
 c(x - {b}, {b}) \le 
c(x - ({b} + \chi_s - \chi_t), {b} + \chi_s - \chi_t),
\]
which, together with \eqref{eqn:sra-opt-update-9},
implies 
\[
 c(\hat{x} - \hat{b}, \hat{b}) \ge 
c(\hat{x} - (\hat{b}- \chi_s + \chi_t), \hat{b} - \chi_s + \chi_t).
\]
 Since $\hat{b}$ is an optimal solution of
(SRA$(x+\chi_i -\chi_j)$)
and $\hat{b} - \chi_s + \chi_t$ is a feasible solution
of (SRA$(x+\chi_i -\chi_j)$),
the vector $\hat{b}- \chi_s + \chi_t$  is also an optimal solution
of (SRA$(x+\chi_i -\chi_j)$),
a contradiction to the choice of $\hat{b}$
since 
$\|(\hat{b} - \chi_s + \chi_t) - {b}\|_1
=  \|\hat{b} - {b}\|_1 - 2$.
 This concludes the proof of the former statement.

 To prove the latter statement, 
we assume $b(N) < B$;
the case $\hat{b}(N)< B$ can be proven in a similar way.
 Assume, to the contrary, that
there exist some $s \in N$
satisfying 
$s \in \suppp(\hat{b}- {b}) \setminus \{i\}$, 
or $s=i$ and $\hat{b}(i)- b(i) \ge 2$.
 Then, we have
\begin{align}
& 
c(\hat x - \hat b,  \hat b) + c(x -  b,   b)   
-   c(\hat x - (\hat b - \chi_s),  \hat b - \chi_s)
- c(x - (b + \chi_s),   b + \chi_s) 
\notag\\
& =
c_s(\hat x(s) -\hat{b}(s), \hat{b}(s)) + c_s(x(s)- {b}(s), {b}(s))
\notag\\
& \qquad
-
 c_s(\hat x(s) - \hat{b}(s)+1, \hat{b}(s)-1)
- c_s(x(s) - {b}(s) - 1, {b}(s)+1)
\notag\\
& \ge 0,
  \label{eqn:sra-opt-update-1}
\end{align}
where the inequality is by  \eqref{eqn:sra-opt-update-7}
in Claim 1. 
 The vector ${b}+ \chi_s$ is a feasible solution of (SRA($x$))
since $b(N) < B$.
 Hence, we have 
\[
 c(x - {b}, {b}) \le c(x - ({b} +\chi_s), {b} + \chi_s),
\]
which, together with \eqref{eqn:sra-opt-update-1},
implies 
\begin{equation}
  \label{eqn:sra-opt-update-2}
  c(\hat{x} - \hat{b}, \hat{b}) \ge 
c(\hat{x} - (\hat{b} - \chi_s), \hat{b} - \chi_s).
\end{equation}
 Since  $\hat{b} - \chi_s$  is a feasible solution
of (SRA$(x+\chi_i -\chi_j)$),
optimality of $\hat{b}$ and 
the inequality \eqref{eqn:sra-opt-update-2} imply that 
$\hat{b} - \chi_s$  is also an optimal solution
of (SRA$(x+\chi_i -\chi_j)$),
a contradiction to the choice of $\hat{b}$
since 
$\|(\hat{b} - \chi_s) - {b}\|_1
=  \|\hat{b} - {b}\|_1 - 1$.
 Hence, the condition (a) holds.
[End of Claim]

\medskip

 We now prove the lemma.
 It is easy to see from Claim 2 that the following properties hold:
\begin{quote}
$\bullet$ 
if $b(N) = \hat{b}(N) < B$.
then 
$\hat{b} \in \{b, b + \chi_i - \chi_j\}$,
\\
$\bullet$ 
if $b(N) <\hat{b}(N) \le B$, then $\hat{b} = b+ \chi_i$,
\\
$\bullet$ 
if $B \ge b(N) > \hat{b}(N)$, then
$\hat{b} = b- \chi_j$.
\end{quote}

 We next consider the case with $b(N) = \hat{b}(N) = B$.
 Then, one of (a) and (b) holds by Claim 2.
 Suppose that (a) holds.
 If $\suppp(\hat{b}-b) = \emptyset$,
then $\hat{b}=b$ follows since $b(N) = \hat{b}(N)$.
 If $\suppp(\hat{b}-b) \ne \emptyset$,
then we have $\suppp(\hat{b}-b) = \{i\}$ and $\hat{b}(i)=b(i)+1$.
 Since $b(N) = \hat{b}(N)$,
there exists a unique element $t$ in 
$\suppm(\hat{b}-b)$ and it satisfies
$i \ne i$  and $\hat{b}(t)=b(t)-1$.
 Hence, we have
$\hat{b}=b$ or $\hat{b}=b + \chi_i - \chi_t$ 
for some $t \in N \{i\}$.
 If the condition (b) holds, then
we can show in a similar way that
$\hat{b}=b$ or $\hat{b}=b + \chi_s - \chi_j$ 
for some $s \in N \setminus \{j\}$.

\subsection{Proof of Proposition \ref{prop:DA-convex-B}}

We consider a variant of the problem (DA), where
the constant $B$ is replaced with a parameter $\alpha$:
\[
 \begin{array}{l|lll}
\mbox{(DA$[\alpha]$)} & \mbox{Minimize }  & 
\displaystyle c(d, b)\\
& \mbox{subject to } & 
\displaystyle  d(N)+b(N) = D+ B,\\
& & 
\displaystyle  b(N) \le  \alpha,\\
& & \ell \le d + b \le u,\\
& &  \bar{d} + \bar{b} -\gamma \1 \le    d + b
 \le \bar{d} + \bar{b} + \gamma \1,\\
& &  d, b \in \Z^n_+.
 \end{array}
\]
 We denote by $\psi(\alpha)$
the optimal value of (DA$[\alpha]$).
 To prove Proposition \ref{prop:DA-convex-B},
it suffices to show the following property of $\psi(\alpha)$.

\begin{lemma}
The value $\psi(\alpha)$ is a convex function in
  $\alpha \in [0, B]$.
\end{lemma}

\begin{proof}
  For $\alpha \in [0, B-2]$,
we show that $\psi(\alpha)+ \psi(\alpha+2) \ge 2 \psi(\alpha+1)$ holds.

 Let $(d,b) \in \Z^n \times \Z^n$ be an optimal solution of (DA$[\alpha]$).
 Also, 
let $(\hat d,\hat b) \in \Z^n \times \Z^n$ be an optimal solution of 
(DA$[\alpha+2]$) that has the minimum value of 
$\|\hat{d} - d\|_1 + \|\hat{b} - b\|_1$.
 Note that we have $c(d,b)= \psi(\alpha)$
and $c(\hat d,\hat b)= \psi(\alpha+2)$.
 By the definition of $\psi$, we have 
$\psi(\alpha) \ge \psi(\alpha+1) \ge \psi(\alpha+2)$.
 Hence, if $(\hat d,\hat b)$ is a feasible solution of (DA$[\alpha+2]$),
then we have $\psi(\alpha+1) \le \psi(\alpha+2)$ and therefore
the inequality $\psi(\alpha)+ \psi(\alpha+2) \ge 2 \psi(\alpha+1)$ follows.
 Therefore, we may assume that $\hat{b}(N) = \alpha+2$ in the following.

 Since $\hat{b}(N)= \alpha + 2 > \alpha = b(N)$
and $\hat{d}(N)+\hat{b}=d(N)+b(N)$, it holds that
$\suppp(\hat{b}- b) \ne \emptyset$ and
$\suppm(\hat{d}- d) \ne \emptyset$.
 We first consider the case where 
 there exists some $i \in \suppp(\hat{b}- b)$
with $i \in \suppm(\hat{d}-d)$.
 Then, Proposition \ref{prop:2multimod} (i) implies that
\[
 c_i(\hat d(i), \hat{b}(i)) + c_i(d(i), {b}(i))  
 \ge
c_i(\hat d(i)+1, \hat{b}(i)-1) + c_i(d(i)-1, {b}(i)+1),  
\]
from which follows that
the vectors $\hat{d}' = \hat{d}+ \chi_i$,
$\hat{b}' = \hat{b}- \chi_i$,
$d'=d-\chi_i$, and $b'=b+\chi_i$ satisfy the inequality
\begin{align*}
 \psi(\alpha+2)+ \psi(\alpha)
 = c(\hat{d},\hat{b})+d(d,b) \ge  c(\hat{d}',\hat{b}')+c(d',b')
 \ge 2\psi(\alpha+1),
\end{align*}
where the last inequality is by the fact that
 $(\hat{d}',\hat{b}')$ and $(d',b')$ are feasible solutions of (DA$[\alpha+1]$).

 We next consider the case where 
 there exists no $i \in \suppp(\hat{b}- b)$
with $i \in \suppm(\hat{d}-d)$,
i.e., we have
$\suppp(\hat{b}- b)\subseteq  N \setminus \suppm(\hat{d}-d)$ 
and therefore $\hat d(i) + \hat{b}(i) > d(i)+ {b}(i)$ holds
for every $i \in \suppp(\hat{b}- b)$.
 Then, we have $\suppm(\hat{d}-d) \subseteq N \setminus \suppp(\hat{b}-b)$
and therefore  $\hat d(j) + \hat{b}(j) < d(j)+ {b}(j)$ holds
for every $j \in \suppm(\hat{d}- d)$.
 By Proposition \ref{prop:2multimod} (ii),
for $i \in \suppp(\hat{b}- b)$ and  $j \in \suppm(\hat{d}-d)$ it holds that
\begin{align*}
 c_i(\hat d(i), \hat{b}(i)) + c_i(d(i), {b}(i))  
& \ge
c_i(\hat d(i), \hat{b}(i)-1) + c_i(d(i), {b}(i)+1),  
\\
 c_j(\hat d(j), \hat{b}(j)) + c_j(d(j), {b}(j))  
& \ge
c_j(\hat d(j)+1, \hat{b}(j)) + c_j(d(j)-1, {b}(j)),  
\end{align*}
from which follows that
the vectors $\hat{d}'' = \hat{d}+ \chi_j$,
$\hat{b}'' = \hat{b}- \chi_i$,
$d''=d-\chi_j$, and $b''=b+\chi_i$ satisfy the inequality
\begin{align*}
 \psi(\alpha+2)+ \psi(\alpha)
 = c(\hat{d},\hat{b})+d(d,b) \ge  c(\hat{d}'',\hat{b}'')+c(d'',b'')
 \ge 2\psi(\alpha+1),
\end{align*}
where the last inequality is by the fact that
 $(\hat{d}'',\hat{b}'')$ and $(d'',b'')$ are feasible solutions 
of (DA$[\alpha+1]$).
\end{proof}

\subsection{Algorithms for (DA)}

\subsubsection{Algorithms}

 We first propose a steepest descent algorithm for (DA).
 By using the fact that the problem (DA) can be reformulated as the minimization
of the M-convex function $f$ given by \eqref{eqn:def-f}, 
we can show that (DA) can be solved by a steepest descent algorithm 
similar to {\sc SteepestDescentDR$'$}.
 Difference from {\sc SteepestDescentDR$'$}
is in the choice of the initial vector
and in the termination condition. 
 In the algorithm below, the initial vector can be any feasible solution
that is bike-optimal, 
and the termination condition is given by a local optimality.
 Here, we say that a feasible solution $(d,b)$ of (DA) is
\textit{bike-optimal} if $b$ is an optimal solution of the problem (SRA$(d+b)$).

\begin{flushleft}
 \textbf{Algorithm} {\sc SteepestDescentDA} 
\\
 \textbf{Step 0:} 
 Set $(d_0, b_0)$ be an arbitrarily chosen
bike-optimal feasible solution, and $k:=1$.
\\
 \textbf{Step 1:} 
 If $c(d',b') \ge c(d_{k-1}, b_{k-1})$
for every $(d', b') \in N(d_{k-1}, b_{k-1}) \cap R$,
then output\\
\noindent
\phantom{Step 1: }
 the solution $(d_{k-1}, b_{k-1})$ and stop.\\
 \textbf{Step 2:} 
 Find $(d', b') \in N(d_{k-1}, b_{k-1}) \cap R$
that minimizes $c(d', b')$.\\
 \textbf{Step 3:} 
 Set $(d_{k}, b_{k}) := (d', b')$ and go to Step 1.
\end{flushleft}

 By applying Corollary \ref{coro:M-steep-descent} 
to $f$ in \eqref{eqn:def-f}
and also using the same analysis in Section 6.1, we obtain the following
time complexity bound.

\begin{theorem}
\label{thm:sda-DA}
The algorithm {\sc SteepestDescentDA}  
outputs an optimal solution in $\Oh(n + \nu \log n)$ time
 with 
\[
 \nu = \min\{\|(d+b) - (d_0+b_0)\|_1 \mid 
(d,b) \mbox{ is an optimal solution of \rm (DA)}
\}.
\]
\end{theorem} 

 We then propose a polynomial-time proximity-scaling algorithm
for (DA).
 For a positive integer $\lambda$ and
a feasible solution $(d,b)$  of (DA), we say that
$(d,b)$ is $\lambda$-optimal if
$c(d',b') \ge c(d,b)$ holds 
for every feasible solution $(d', b')$  of {\rm (DA)} with
 $d'(i)-d(i) \in \{0, + \lambda, - \lambda\}$ 
and  $b'(i)-b(i) \in \{0, + \lambda, - \lambda\}$. 

 A $\lambda$-optimal solution can be obtained by finding an optimal solution
of the following problem:
\[
 \begin{array}{l|lll}
\mbox{\rm (DA$(\lambda)$)} & \mbox{\rm Minimize }  & 
\displaystyle c(d, b)\\
& \mbox{\rm subject to } & 
\displaystyle  d(N)+b(N) = D+ B,\\
& & 
\displaystyle  b(N) \le  B,\\
& &  \ell \le d + b \le  u,\\
& &  d, b \in \Z^n_+,\\
& &  d(i)-\check{d}(i)  , b(i) - \check{b}(i)
\mbox{ are integer multiple of }\lambda \mbox{ for }i \in N,
 \end{array}
\]
where $(\check{d}, \check{b})$ is some (fixed)
feasible solution of (DA).
 Note that for $i \in N$, the function
\[
 c_i^\lambda(\eta, \zeta) =
c_i(\lambda \eta + \check{d}(i), \lambda \zeta + \check{b}(i))
\]
is also a multimodular function in $(\eta,\zeta)$, 
the problem (DA$(\lambda)$) has the same comibnatorial structure as (DA),
and therefore any algorithm for  (DA) can be applied 
to (DA$(\lambda)$).
 Our proximity-scaling algorithm is based on this fact and
the following proximity theorem for (DA):

\begin{theorem}
\label{thm:bdr-proximity}
 Let $\lambda$ be a positive integer with $\lambda \ge 2$,
and $(d,b) \in \Z^n \times \Z^n$ be a $\lambda$-optimal solution of {\rm (DA)}.
Then, there exists some optimal solution 
$(d^*,b^*) \in  \Z^n \times \Z^n$ of {\rm (DA)}
such that 
\[
 \|(d^*+b^*)  - (d+b)\|_1 \le 8 \lambda n.
\]
\end{theorem}
\noindent
 Proof is given later in this subsection.

\begin{flushleft}
 \textbf{Algorithm} {\sc ProximityScalingDA} 
\\
 \textbf{Step 0:} 
 Let
$(d_0, b_0)$ be an arbitrarily chosen  feasible solution
of (DA) and $x_0 = d_0+b_0$.
 Set $\lambda = (D+B)/4n$ and  $p:=1$.
\\
 \textbf{Step 1:} 
 Let $b_{p-1}' \in \Z^n$ be a vector such that
 $(x_{p-1} - b_{p-1}', b_{p-1}')$ is a bike-optimal solution of
(DA$(\lambda)$).
\\
 \textbf{Step 2:} 
 Apply the  algorithm {\sc SteepestDescentDA} 
to  (DA$(\lambda)$) with the initial solution $(x_{p-1} - b_{p-1}', b_{p-1}')$
to find a $\lambda$-optimal solution $(d_p, b_p)$.
\\
 \textbf{Step 3:} 
 If $\lambda = 1$, then output $(d_p, b_p)$ and stop.
 Otherwise, set $x_p = d_p+b_p$, $\lambda := \lfloor (\lambda/2) \rfloor$,
$p:=p+1$, and go to Step 1.
\end{flushleft}

 We analyze the time complexity of the algorithm  {\sc ProximityScalingDA}. 
 The definition of the initial $\lambda$ in Step 0 implies that
there exists a $\lambda$-optimal solution $(d,b)$ 
with $\|(d+b) - x_0\|_1 \le 8 \lambda n$.
 Also, in the $p$-th iterations with $p\ge 2$, 
Theorem \ref{thm:bdr-proximity} implies that
there exists a $\lambda$-optimal solution $(d,b)$ 
with $\|(d+b) - x_{p-1}\|_1 \le 8 \lambda n$.
 Hence, it follows from Theorem \ref{thm:sda-DA} that
each iteration, except for Step 1, can be done in
$\Oh(n\log n)$ time.
 We can also show in a similar way that in Step 1, the vector 
$b_{p-1}'$ can be computed by using a variant of steepest descent algorithm
with the initial vector $b_{p-1}$,
and prove that its running time is $\Oh(n \log n)$.
 Hence, each iteration of the algorithm runs in $\Oh(n \log n)$.
 Since the number of iterations is $\Oh(\log ((D+B)/n))$,
we obtain the following bound for the algorithm {\sc ProximityScalingDA}. 

\begin{theorem}
The algorithm {\sc ProximityScalingDA}
finds an optimal solution of the problem {\rm (DA)}  
in $\Oh(n \log n \log ((D+B)/n))$ time.
\end{theorem}

\subsubsection{Proof of Theorem \ref{thm:bdr-proximity}}

Let $(d^*,b^*)$ be an optimal solution of {\rm (DA)}
that minimizes the value $\|d^* - d\|_1 + \|b^* - b\|_1$.
 We prove that $(d^*,b^*)$ satisfies the inequality
$ \|x^*  - x\|_1 \le 8 \lambda n$
with $x =d+b$ and $x^* = d^*+b^*$.

 In the proof we consider the following six sets.
\begin{align}
& \label{eqn:prox-cond-1} 
I_1 = \{i \in N  \mid
d(i)- d^*(i) \ge \lambda,\ b(i) - b^*(i) \le -\lambda \},
\\
& \label{eqn:prox-cond-2} 
I_2 = \{i \in N  \mid
d(j)- d^*(j) \le -\lambda,\ b(j) - b^*(j) \ge \lambda\},
\\
& \label{eqn:prox-cond-3} 
I_3 = \{i \in N  \mid
x(i)- x^*(i) \ge \lambda,\  d(i) - d^*(i) \ge \lambda\},
\\
& \label{eqn:prox-cond-4} 
I_4 = \{i \in N  \mid
x(j)- x^*(j) \le -\lambda,\  d(j) - d^*(j) \le -\lambda\},
\\
& \label{eqn:prox-cond-5} 
I_5 = \{i \in N  \mid
x(i)- x^*(i) \ge \lambda,\  b(i) - b^*(i) \ge \lambda\},
\\
& \label{eqn:prox-cond-6} 
I_6 = \{i \in N  \mid 
x(j)- x^*(j) \le -\lambda,\  b(j) - b^*(j) \le -\lambda\}.
\end{align}

\begin{lemma}\ \\
\label{lem:setI-1}
{\rm (i)} At least one of $I_1$ and $I_2$ is an empty set.\\
{\rm (ii)} If $b^*(N) < B$ then $I_2 = \emptyset$ holds;
if $b(N) - B \le -\lambda$ then $I_1 = \emptyset$ holds.
\end{lemma}

\begin{proof}
 We first prove (i).
 Assume, to the contrary, that both of
 $I_1 \ne \emptyset$ and $I_2 \ne \emptyset$ hold.
 Then, there exist distinct $i,j \in N$ such that
 \begin{align*}
 &
 d(i)- d^*(i) \ge \lambda,\  b(i) - b^*(i) \le - \lambda,\qquad
 d(j)- d^*(j) \le - \lambda, \  b(j) - b^*(j) \ge  \lambda.
 \end{align*}
 We consider the pair of vectors
 $(d', b') \equiv (d - \lambda\chi_i + \lambda \chi_j, b + \lambda\chi_i - \lambda \chi_j)$,
 which is  a feasible solution of (DR-AP$(\lambda)$)
 since
 $d+b = d' + b'$ and $b(N)= b'(N)$.
 We show below that $c(d',b') < c(d,b)$ holds,
 a contradiction to the choice of $(d,b)$.

 By Proposition \ref{prop:2multimod} (i), we have
 \begin{align*}
 &  c_i(d(i), b(i))+c_i(d^*(i), b^*(i)) 
 \ge c_i(d(i)-1, b(i)+1)+c_i(d^*(i)+1, b^*(i)-1),\\
 &  c_j(d(j), b^*(j))+c_i(d^*(j), b^*(j)) 
 \ge c_i(d(j)+1, b(j)-1)+c_i(d^*(j)-1, b^*(j)+1).
 \end{align*}
 This implies
 \begin{equation}
 \label{eqn:prox-1}
  c(d,b) + c(d^*, b^*) \ge 
 c(d-\chi_i+\chi_j, b+\chi_i - \chi_j) + 
 c(d^*+\chi_i-\chi_j, b^*-\chi_i + \chi_j).
 \end{equation}
 Note that $(d'', b'') = (d^*+\chi_i-\chi_j, b^*-\chi_i + \chi_j)$
 is also a feasible solution of (DR-AP) since
$d''+b'' = d^*+b^* $ and $b''(N)=b^*(N)$.
 Since $(d^*+\chi_i-\chi_j, b^*-\chi_i + \chi_j)$ satisfies
 \[
 \|(d^*+\chi_i-\chi_j) - d\|_1 + \|(b^*-\chi_i + \chi_j) - b\|_1
 < \|d^* - d\|_1 + \|b^* -b \|_1, 
 \]
 we have
 \[
 c(d^*, b^*) <c(d^*+\chi_i-\chi_j, b^*-\chi_i + \chi_j),
 \]
 which, together with \eqref{eqn:prox-1}, implies
 $c(d,b) > c(d-\chi_i+\chi_j, b+\chi_i - \chi_j)$.

 In a similar way, we can also prove
 the inequalities
 \[
  c(d-\chi_i+\chi_j, b+\chi_i - \chi_j) 
 >  c(d- 2 \chi_i+ 2 \chi_j, b+ 2 \chi_i - 2 \chi_j) 
 >  \cdots
 >
 c(d- \lambda\chi_i+ \lambda \chi_j, b+ \lambda \chi_i - \lambda\chi_j),
 \]
 from which $c(d,b) > c(d', b')$ follows.

 Proof of (ii) is similar to (i) and omitted.
\end{proof}

\begin{lemma}
\label{lem:setI-2} 
At least one of $I_3 = \emptyset$ and  $I_4 = \emptyset$ holds.
\end{lemma}

\begin{proof}
 Proof is similar to that for Lemma \ref{lem:setI-1}
and omitted.
\end{proof}

\begin{lemma}
\label{lem:setI-3} 
At least one of $I_5 = \emptyset$ and 
$I_6 = \emptyset$ holds.
\end{lemma}

\begin{proof}
 Proof is similar to that for Lemma \ref{lem:setI-1}
and omitted.
\end{proof}

\begin{lemma}
\label{lem:setI-4} 
\ \\
{\rm (i)}
At least one of 
$I_4$, $I_5$, and $I_1$ is an empty set.
\\
{\rm (ii)}
If $b^*(N) < B$ then
at least one of $I_4$ and $I_5$ is an empty set.
\\
{\rm (iii)}
At least one of 
$I_3$, $I_6$, and $I_2$ is an empty set.
\\
{\rm (iv)}
 If $b(N) - B \le -\lambda$ then
at least one of $I_3$ and $I_6$ is an empty set.
\end{lemma}

\begin{proof}
 We prove (i) only.
 Assume, to the contrary, that
all of the sets $I_4$, $I_5$, and $I_1$ 
are nonempty, and let
$i \in I_4$, $j \in I_5$, and $s \in I_1$.
 Then, elements $i,j,s$ are distinct 
 by the definitions of $I_4, I_5$, and $I_1$.
 We denote 
 \begin{align*}
 &  (d', b') = 
 (d + \lambda\chi_i  - \lambda \chi_s ,
 b - \lambda \chi_j + \lambda \chi_s),
 \\
 & (d'', b'') = (d^* -\chi_i + \chi_s, b^* + \chi_j - \chi_s).
 \end{align*}
 Since $(d',b')$ and $(d'', b'')$ satisfy
 \[
 d'(P)+b'(P) = d(P)+b(P), \ b'(N)=b(N), \quad
 d''(P)+b''(P) = d^*(P)+b^*(P), \ b''(N)=b^*(N),
 \]
 $(d',b')$ (resp.,  $(d'', b'')$) is a feasible solution
 of (DR-AP$(\lambda)$) (resp., (DR-AP)).
 Using this fact, we can derive a contradiction as in Lemma \ref{lem:setI-1}.
\end{proof}

\begin{lemma}
\label{lem:i4i6-bound}
 We have $\| x-x^* \|_1 \le 4 \lambda n$
if at least one of the followng two conditions holds:
\begin{quote}
{\rm (a)}  $I_3= I_5  = \emptyset$,
\qquad
{\rm (b)}  $I_4= I_6  = \emptyset$.
\end{quote}
\end{lemma}

\begin{proof}
 Suppose that  $I_4= I_6  = \emptyset$ holds.
 Then, we have $x(i) - x^*(i) \ge - 2\lambda$ for every $i \in N$.
 Let $N_- = \suppm(x-x^*)$.
 Since $x(N)-x^*(N)=0$, we have
\begin{align*}
\| x-x^* \|_1 
& = [ x(N \setminus N_-) - x^*(N \setminus N_-)]  +[ x^*(N_-) - x(N_-)] \\
& =  [ x(N) - x^*(N)]  + 2 [ x^*(N_-) - x(N_-)] \\
& = 4 \lambda |N_-| \le 4\lambda n.
\end{align*}
 Proof for the case with $I_3= I_5  = \emptyset$ is similar.
\end{proof}

\begin{lemma}
\label{lem:i2i4i5-bound}
 We have $\| x-x^* \|_1 \le 8 \lambda n$
if at least one of the followng two conditions holds:
\begin{quote}
{\rm (a)}  $I_2= I_4= I_5  = \emptyset$ and $b(N)-b^*(N) > - \lambda$,
\\
{\rm (b)}  $I_1= I_3= I_6  = \emptyset$ and $b(N)-b^*(N) <  \lambda$.
\end{quote}
\end{lemma}

\begin{proof}
 We consider the case where (a) holds, and 
show that $\| d-d^* \|_1 \le 4 \lambda n$
and $\| b-b^* \|_1 \le 4 \lambda n$ hold,
which implies
\[
\| x-x^* \|_1  \le \| d-d^* \|_1+ \| b-b^* \|_1
\le 8 \lambda n.
\]

 Since $I_2= I_4= I_5  = \emptyset$, it holds that
\begin{equation}
\label{eqn:i2i4i5-bound-1}
d(i) - d^*(i) \ge -2 \lambda, \quad b(i) - b^*(i) \le 2 \lambda
\qquad (i \in N).
\end{equation}
 Since $b(N)-b^*(N) > - \lambda$
and $x(N) - x^*(N)=0$, we have $d(N)-d^*(N) < \lambda$.

 To prove the inequality $\| d-d^* \|_1 \le 4 \lambda n$,
let $H = \suppm(d-d^*)$.
 If $H = N$, then
we have $d(i) - d^*(i) < 0$ for every $i \in N$, implying that
\[
\| d-d^* \|_1
= \sum_{i\in N}|d(i)-d^*(i)| 
=  \sum_{i\in N}[d^*(i)-d(i)]
=  d^*(N)-d(N)< \lambda \le 4\lambda n.
\]
 If $H \ne N$, then
we have
\begin{align*}
\| d-d^* \|_1
= \sum_{i\in N'}|d(i)-d^*(i)| 
& = [ d(N \setminus H) - d^*(N \setminus H)]  +[ d^*(H) - d(H)] \\
& =  [ d(N) - d^*(N)]  + 2 [ d^*(H) - d(H)] \\
& < \lambda + 4 \lambda |H| \le 4\lambda n,   
\end{align*}
where the first inequality is by  $d(N)-d^*(N)< \lambda$ and 
$d(i) - d^*(i) \ge - 2\lambda$ for $i \in N$,
and the second inequality is by $|H| < n$.
 The inequality $\| b-b^* \|_1 \le 4 \lambda n$ can be proved similarly
by using the inequalities $b(N)-b^*(N)> - \lambda$ and 
$b(i) - b^*(i) \le  2\lambda$ for $i \in N$.
\end{proof}

\begin{lemma}
We have $\| x-x^* \|_1 \le 8 \lambda n$.
\end{lemma}

\begin{proof}
 By Lemmas \ref{lem:setI-2}  and \ref{lem:setI-3}, 
we have the following four possibilities:
\begin{quote}
(Case 1)  $I_4 = I_6= \emptyset$, \quad
(Case 2)  $I_3  = I_5  = \emptyset$, \\
(Case 3)   $I_4  = I_5  = \emptyset$, $I_3  \ne \emptyset,\  I_6  \ne \emptyset$, 
(Case 4)  $I_3   = I_6  = \emptyset$,
 $I_4  \ne \emptyset,\  I_5  \ne \emptyset$. 
\end{quote}
 If Case 1 or 2 holds, then we have $\| x-x^* \|_1 \le 8 \lambda n$
by Lemma \ref{lem:i4i6-bound}.
 Below we give proofs for Cases 3 and 4.

[Proof for Case 3] \quad
 By Lemma \ref{lem:setI-4} (iii) and (iv),
we have $I_2 = \emptyset$ and $b(N) - B > - \lambda$;
the second inequality implies $b(N) - b^*(N) > - \lambda$
since $b^*(N) \le B$.
 Hence,  we have $\| x-x^* \|_1 \le 8 \lambda n$
by Lemma \ref{lem:i2i4i5-bound}.

[Proof for Case 4] \quad
 By Lemma \ref{lem:setI-4} (i) and (ii),
we have $I_1 = \emptyset$ and $b^*(N) = B$;
the second equation implies $b(N) - b^*(N) <  \lambda$
since $b(N) \le B$.
 Hence,  we have $\| x-x^* \|_1 \le 8 \lambda n$
by Lemma \ref{lem:i2i4i5-bound}.
\end{proof}


\begin{thebibliography}{99}

\bibitem{FHS2017}
D. Freund, S.G.~Henderson, and D.B.~Shmoys.  
Minimizing multimodular functions and allocating capacity in bike-sharing systems. 
In 
\textit{Proc.~19th IPCO},
LNCS  10328, pages 186--198. Springer, Berlin, 2017.

\bibitem{FHS2017-2}
D. Freund, S.G.~Henderson, and  D.B.~Shmoys,  
Minimizing multimodular functions and
	allocating capacity in bike-sharing systems.
preprint, arXiv:1611.09304, 2017.



\bibitem{Hoch94}
D.~S. Hochbaum,
Lower and upper bounds for the allocation problem and other nonlinear
optimization problems.
\textit{Math. Oper. Res.}, 19:390--409, 1994.


\bibitem{IK88}
T. Ibaraki and N. Katoh:
\textit{Resource Allocation Problems: Algorithmic Approaches.}
MIT Press, Cambridge, MA, 1988.





\bibitem{IMM05}
S. Iwata, S. Moriguchi, and K. Murota.
A capacity scaling algorithm for M-convex submodular flow.
\textit{Math. Program.},  103:181-202, 2005.

\bibitem{IS03}
S. Iwata and M. Shigeno. 
Conjugate scaling algorithm for Fenchel-type duality 
in discrete convex optimization. 
\textit{SIAM J. Optim.}, 13:204--211, 2003.

\bibitem{MoriMuro2018}
S. Moriguchi and K. Murota.
On fundamental operations for multimodular functions.
preprint,  arXiv:1805.04245, 2018.

\bibitem{MMS2002}
S. Moriguchi, K. Murota, and A. Shioura.
Scaling algorithms for M-convex function minimization.
\textit{IEICE Transactions on Fundamentals}, E85-A: 922--929, 2002.

\bibitem{Mstein} 
K. Murota.
Convexity and Steinitz's exchange property.
\textit{Adv. Math.}, 124:272--311, 1996. 

\bibitem{Mdca} 
K. Murota.
Discrete convex analysis.
\textit{Math. Program.}, 83:313--371, 1998.

 \bibitem{Murota03book}
K. Murota.
 \textit{Discrete Convex Analysis}. 
 SIAM, Philadelphia, 2003.

\bibitem{Murota03}
K. Murota.
 On steepest descent algorithms for discrete convex functions.
\textit{SIAM J. Optim.}, 14:699--707, 2003. 

\bibitem{Murota05}
K. Murota.
Note on multimodularity and L-convexity. 
\textit{Math. Oper. Res.}, 30:658-661, 2005. 

\bibitem{MS99}
K. Murota and A. Shioura.
M-convex function on generalized polymatroid.
\textrm{Math. Oper. Res.}, {24}:95--105,  1999.

\bibitem{Sh1998}
 A. Shioura. 
Minimization of an M-convex function.
\textit{Discrete Appl. Math.}, 84:215--220, 1998.


\bibitem{Shioura04}
A. Shioura.
Fast scaling algorithms for M-convex function minimization 
with application to the resource allocation problem.
\textit{Discrete Appl. Math.}, 134:303--316, 2004. 

\bibitem{T-scal}
A. Tamura.
Coordinatewise domain scaling algorithm for
M-convex function minimization.
\textit{Math. Program.}, 102:339--354, 2005.

\end{thebibliography}
\end{document}